\documentclass[11pt]{article}

\usepackage[margin=1in]{geometry}
\usepackage{amsthm}
\usepackage{amsmath}
\usepackage{amssymb}
\usepackage{mathrsfs}
\usepackage{euscript}
\usepackage{graphicx}
\usepackage{pgf,tikz}
\usepackage{float}
\usepackage{mathtools}
\usepackage{color}
\usepackage[noadjust]{cite}
\usepackage{enumerate}
\usepackage{multicol}
\usepackage{centernot}
\usetikzlibrary{arrows}
\usepackage{comment}
\usepackage{hyperref}
\usepackage{xcolor}
\usepackage{tikz}
\usepackage{ulem}
\theoremstyle{definition}
\newtheorem{theorem}{Theorem}[section]
\newtheorem{corollary}[theorem]{Corollary}
\newtheorem{proposition}[theorem]{Proposition}
\newtheorem{lemma}[theorem]{Lemma}

\newtheorem{statement}[theorem]{Statement}

\theoremstyle{definition}
\newtheorem{example}[theorem]{Example}
\newtheorem{definition}[theorem]{Definition}
\newtheorem{question}{Question}

\newtheorem{remark}[theorem]{Remark}
\definecolor{cInf}{RGB}{226,102,102}   % infinitely recurrent
\definecolor{cPer}{RGB}{ 10,110,255}   % persistently recurrent
\definecolor{cUni}{RGB}{164,  0,220}   % uniformly recurrent
\definecolor{cReg}{RGB}{245, 105, 30}   % regularly recurrent
\definecolor{cQ}{RGB}{235,150,  0}     % Q(k)->infty
\definecolor{cRen}{RGB}{150,100, 45}   % infinitely renormalizable
\definecolor{cFib}{RGB}{150,150,150}   % Fibonacci like
\definecolor{cLong}{RGB}{ 95,185, 65}  % long-branched
\definecolor{c5}{RGB}{ 20, 20,140}  % wild attractor
\definecolor{cMono}{RGB}{  0,125, 70}  % not critically monotonic
\definecolor{cSAM}{RGB}{ 70,205,195}   % strange adding machine
\definecolor{cMin}{RGB}{255,  0,0}   % minimal homeomorphism
\definecolor{cSlow}{RGB}{  0,190,235}  % slowly recurrent
\definecolor{cCE}{RGB}{100,100,100} % Collet-Eckmann

\newcommand{\N}{\mathbb{N}}

\newcommand{\RR}{\mathcal{R}}
\renewcommand{\epsilon}{\varepsilon}
\newcommand{\invlim}{\varprojlim}

\newcommand{\Lori}[1]{\textcolor{blue}{#1}}
\newcommand{\Jernej}[1]{\textcolor{teal}{#1}}

\title{Recurrence, symbolic dynamics, and wild attractors for unimodal maps}
\author{
Lori Alvin\thanks{ORCID: https://orcid.org/0000-0001-7557-2669.\\
Department of Mathematics, Furman University, Greenville, SC 29613, USA.\\
Email: \texttt{lori.alvin@furman.edu}}
\and
Jernej \v{C}in\v{c}\thanks{ORCID: https://orcid.org/0000-0001-8516-6023.\\
Department of Mathematics and Computer Science, Faculty of Natural Sciences and Mathematics,
University of Maribor, Koro\v{s}ka 160, 2000 Maribor, Slovenia; and
Abdus Salam International Centre for Theoretical Physics (ICTP), Trieste, Italy.\\
Email: \texttt{jernej.cinc@um.si}
}}

\begin{document}
	\maketitle

\begin{abstract}
In this paper, we construct families of unimodal maps with wild Cantor attractors and study, via symbolic dynamics, the interplay between recurrence properties of the critical orbit. Using kneading and co-kneading techniques, we provide a symbolic characterization of persistent recurrence in terms of cutting and co-cutting times, and relate it to other recurrence conditions, including the Collet--Eckmann condition. As an application, we obtain uncountable families of non-renormalizable unimodal maps of sufficiently high critical order admitting wild attractors. We also analyze the structure of postcritical dynamics in recurrent regimes, constructing examples with embedded odometers and with minimal homeomorphic dynamics on the critical $\omega$-limit set, and showing that a minimal homeomorphism together with regular recurrence does not imply conjugacy to an adding machine. Finally, we prove that longbranched combinatorics is incompatible with persistent recurrence, thereby clarifying the boundary between different recurrence regimes.
\end{abstract}

\begingroup
\renewcommand{\thefootnote}{}
\addtocounter{footnote}{-1}
\endgroup

\tableofcontents

\section{Introduction}

We study the relation between recurrence properties of the critical orbit and the existence of wild Cantor attractors for unimodal maps. Wild attractors (also called absorbing Cantor attractors) were first constructed by Bruin, Keller, Nowicki, and van Strien \cite{BKNvS1996}, who showed that certain non-renormalizable unimodal maps admit a Cantor attractor that attracts Lebesgue-almost every orbit while failing to be a topological attractor. Subsequent work of Bruin \cite{Bruin_Absorbing_Cantor} showed that Fibonacci-type combinatorics provide a natural setting for their existence. More broadly, the theory of one-dimensional dynamics shows that the behavior of a unimodal map is largely governed by the recurrence properties of the critical orbit, a perspective developed in Lyubich’s work on quadratic maps \cite{Lyubich2002} and in the classification theory of Avila--Lyubich \cite{AvilaLyubich2006} and Kozlovski--Shen--van Strien \cite{KSVS2007}.

A central theme in this setting is the distinction between recurrence regimes associated with weak expansion and those associated with strong expansion. On the one hand, sufficiently strong growth along the critical orbit leads to stochastic behavior, typically expressed through the Collet--Eckmann condition and the existence of absolutely continuous invariant measures. On the other hand, in the absence of such expansion, recurrence may lead to Cantor attractors, including wild attractors. This makes it natural to ask which recurrence conditions are genuinely relevant for the existence of wild attractors, and how they are reflected in the symbolic combinatorics of the critical orbit.

The main purpose of this paper is to develop a symbolic framework for comparing such recurrence conditions and to apply it to the study of wild attractors and related postcritical dynamics. Our first main result gives a symbolic characterization of persistent recurrence in terms of return, cutting, and co-cutting times, building on kneading theory in the sense of Milnor--Thurston \cite{MilnorThurston88} and on the kneading and co-kneading techniques developed by Hofbauer and Keller \cite{Hofbauer,HofbuerKeller}. We then use this framework in several directions. First, we construct uncountable families of non-renormalizable unimodal maps of sufficiently high critical order that admit wild Cantor attractors, obtaining in particular dense families in the corresponding parameter space. Second, we analyze the postcritical Cantor dynamics arising in recurrent regimes, including examples with embedded odometers and examples for which the restriction to the critical $\omega$-limit set is a minimal homeomorphism but not conjugate to an adding machine. Third, we clarify the position of longbranched and slowly recurrent combinatorics within this picture.

The notion of persistent recurrence was introduced by Blokh and Lyubich \cite{Blokh_Lyubich_Measurable_Dynamics_of_S_unimodal} as a real one-dimensional counterpart of Yoccoz’ puzzle condition \cite{Yoccoz}; see also \cite{KSVS2007}. In the wild-attractor setting, persistent recurrence plays a distinguished role: the known constructions of wild attractors occur in persistently recurrent regimes, and the available structure theory indicates that the existence of a wild attractor forces strong recurrence of the critical orbit. Our first main theorem, Theorem~\ref{Theorem: Main Symbolic Characterization}, provides a symbolic characterization of persistent recurrence. We also give symbolic characterizations of other recurrence conditions relevant to wild attractors, including critical monotonicity and Bruin’s condition (C5); see Theorems~\ref{Theorem: C6 characterization} and \ref{Theorem: C5 characterization}. As a byproduct, we correct examples in \cite{Bruin_Absorbing_Cantor} that were intended to separate these conditions but are not shift-maximal, and replace them by valid kneading examples; see Examples~\ref{Example: Bad C6 not C5} and \ref{Example: Fixed C7 not C6}. We also correct an error in \cite{FPEP} by providing a corrected proof of \cite[Theorem 4.13]{FPEP}; see Theorem~\ref{Theorem: fixed FPEP endpoint characterization}.

A principal application of our symbolic approach is a new description of families of unimodal maps with wild attractors. To our knowledge, the only previously known general condition guaranteeing the existence of wild attractors for sufficiently high critical order is the one given by Bruin \cite{Bruin_Absorbing_Cantor}. Using a different symbolic construction, inspired in part by the example of Li and Shen \cite{LiShen}, we construct an uncountable family of unimodal maps with wild attractors that approximate, arbitrarily well in the supremum metric, any non-renormalizable unimodal map in the logistic family. In this sense we obtain a dense set of unimodal maps with wild attractors for sufficiently high critical order; see Corollary~\ref{cor:Wild attractors dense}. As a further consequence, we obtain uncountably many pairwise non-conjugate $S$-unimodal maps $f$ such that $f|_{\omega(c)}$ is conjugate to a prescribed odometer and $f$ has a wild Cantor attractor; see Corollary~\ref{cor:Wild attractors odometer}.

We also study how recurrence interacts with embedded odometers and minimal postcritical dynamics. Embedded odometers arise naturally in unimodal dynamics both in infinitely renormalizable and in non-renormalizable settings; see, for example, \cite{BlockKeeslingMisiurewicz,Jones2009,Alvin_ftcp,Alvin2021,BruinKellerStPierre1997}. Since a dynamical system is conjugate to an adding machine if and only if it is minimal and every point is regularly recurrent \cite{BlockKeesling2004}, it is natural to ask whether regular recurrence of the critical point forces odometer dynamics whenever the postcritical system is a minimal homeomorphism. We show that this is not the case. In Section~\ref{sec: regularly recurrent minimal homeomorphism not odometer} we construct unimodal maps with regularly recurrent critical point for which $f|_{\omega(c)}$ is a minimal homeomorphism but is not conjugate to an adding machine; see Theorem~\ref{thm: regularly recurrent minimal not odometer}.

Another related theme of the paper is the relationship between recurrence and expansion. From the dynamical point of view, one expects incompatibility between odometer-type postcritical behavior and the Collet--Eckmann condition: the former is minimal and equicontinuous, whereas the latter expresses exponential growth along the critical orbit. In Section~\ref{Section:slow recurrence} we study slowly recurrent kneading sequences for $S$-unimodal maps. We construct an example of a slowly recurrent kneading sequence corresponding to an $S$-unimodal map $f$ for which $f|_{\omega(c)}$ is an odometer; see Example~\ref{Example: slowly recurrent odometer}. We also reprove, by symbolic means and without ergodic-theoretic arguments, a theorem of Bruin \cite[Theorem 2]{Bruin_QuasiSymmetry}, showing that persistent recurrence excludes slow recurrence; see Proposition~\ref{Proposition: Persitently recurrent not slowly recurrent}. In addition, Proposition~\ref{prop: linearly recurrent not slowly recurrent} shows that linearly recurrent kneading sequences are never slowly recurrent.

Finally, we study longbranched unimodal maps and their relation to recurrence. Longbranchedness corresponds to a bounded-geometry regime in which the Hofbauer tower domains remain uniformly large; equivalently, the kneading map $Q$ is bounded. In Section~\ref{section: Longbranchdness} we prove that longbranchedness is incompatible with persistent recurrence: if $Q$ is bounded and the critical point is infinitely recurrent (i.e. recurrent but not periodic), then it is not persistently recurrent; see Theorem~\ref{thm:longbranched not persistently recurrent}. As a consequence, in the longbranched infinitely recurrent regime the set of folding points strictly contains the set of endpoints in the inverse limit, correcting the statement in \cite{FPEP} that this difference might be empty. We also show that boundedness of $Q$ does not rule out complicated co-kneading behavior, and we construct both non-uniformly recurrent and uniformly recurrent longbranched examples. In particular, our Sturmian-type examples correct \cite[Proposition 2]{Outershark}: in that setting the restriction to the critical $\omega$-limit set is minimal but not a homeomorphism; see Proposition~\ref{prop:rotnumber}.

The paper concludes with a diagram summarizing the inclusion relations among the recurrence conditions considered here and with open questions suggested by our study.

\section{Preliminaries}

\subsection{Unimodal maps}

Let $\N:=\{1,2,3,\ldots\}$. A \emph{unimodal map} is a continuous map  $f \colon [0,1] \to [0,1]$
for which there exists a point \( c \in (0,1) \) such that \( f|_{[0,c]} \) is strictly increasing and 
\( f|_{[c,1]} \) is strictly decreasing. The point \( c \) is called the \emph{critical point} (sometimes also called the \emph{turning point}), and for all
\( i \in \mathbb{N} \) we set \( c_i = f^{i}(c) \). 

When \( f \) is \( C^3 \) its \emph{Schwarzian derivative} is defined by
\[
Sf(x)=\frac{D^3f(x)}{Df(x)}-\frac32\left(\frac{D^2f(x)}{Df(x)}\right)^2<0
\]
whenever \(Df(x)\neq0\).

The two most common families of unimodal maps with negative Schwarzian derivative are the \emph{symmetric tent family} and the \emph{logistic family}.  
A symmetric tent map \( T_s \colon [0,1] \to [0,1] \) with parameter \( s \in [0, 2] \) is given by
\[
T_s(x) = 
\begin{cases}
s x, & x < \tfrac12, \\
s (1 - x), & x \ge \tfrac12.
\end{cases}
\]
The logistic map \( g_a \colon [0,1] \to [0,1] \) with parameter \( a \in [0,4] \) is defined by
\[
g_a(x) = a x (1 - x).
\]

A unimodal map $f:[0,1]\to [0,1]$ is called \emph{$S$-unimodal} provided $f$ is $C^3$, has negative Schwarzian derivative, has non-flat turning point (i.e. finite critical order), and $|f'(0)| > 1$. We note that maps in the symmetric tent family are not $S$-unimodal, whereas maps in the logistic family with \(a>1\) are $S$-unimodal.
For the remainder of this paper, we assume that \( f \) is a unimodal map with
\( c_2 < c < c_1 \) and \( c_2 < c_3 \). Then the interval \([c_2, c_1]\) is forward invariant under $f$ and is called the \emph{dynamical core} of \( f \); the interesting dynamics for such a unimodal map take place in the core.

For a unimodal map $f$ and a point $x\in [0,1]$, we may define the \emph{itinerary of $x$ under $f$} by $I(x) = I_0I_1I_2\cdots$ where
\[I_j = \begin{cases}
    1 & \text{ if } f^j(x)>c,\\
    0 & \text{ if } f^j(x)<c,\\
    C & \text{ if } f^j(x) = c.
\end{cases} 
\]
The \emph{kneading sequence of $f$}, denoted $\mathcal{K}(f)$, is the itinerary of the critical value \(c_1=f(c)\), that is
\[
\mathcal K(f)=I(c_1).
\] 
We use the \emph{parity lexicographical ordering} to compare itineraries: given two itineraries $v\neq w$, find the first position in which $v$ and $w$ differ and compare that position using the ordering $0\prec C \prec 1$ if the number of 1s preceding that position is even (i.e., has \emph{even parity}) and $1\prec C \prec 0$ otherwise. Let $\nu$ be either an infinite sequence of 1s and 0s or a finite sequence of 1s and 0s ending in $C$. We say $\nu$ is \emph{shift-maximal} provided $\sigma^k(\nu) \preceq \nu$ for all $k\in \N$, where $\sigma$ is the \emph{shift map}. Every unimodal map has a shift-maximal kneading sequence, and every shift-maximal sequence corresponds to a unique (up to topological conjugacy) unimodal map.

We will also need the notion of the contraction principle. The contraction principle holds  for maps with no wandering intervals and no attracting periodic orbits \cite{ Blokh_Lyubich_Measurable_Dynamics_of_S_unimodal, deMelo_vanStrien}.

\begin{definition}[{\bf Contraction principle}] For every $\varepsilon > 0$ there exists $\delta > 0$ such that for every 
$n > 0$ and every interval $J$ of length $|J| > \varepsilon$, 
$|f^{n}(J)| > \delta$. 
{In particular, the iterates \(f^n|_J\) cannot be homeomorphisms for every \(n\).}
\end{definition}

Throughout the paper we assume that \( f \) admits no periodic attractor and is not renormalizable.
\begin{definition}
    Map \( f \) is called \emph{renormalizable of period} \( n > 1 \) if there exists an interval \( J \subset [c_2, c_1] \) containing \( c \) such that \( f^n(J) \subset J \). The maximal interval with this property is called a \emph{restrictive} interval.
    \end{definition}

\subsection{Cutting times and kneading map} For a more detailed account of the notions introduced in this subsection, see \cite{Bruin94}.

Let \( f^n \) denote the \emph{$n$-th iterate of \( f \)}. {Let \(J\subset[0,1]\) be the maximal (by inclusion relation) interval such that \(c\in\partial J\) and \(f^n|_J\) is monotone,} then  
\( J \) is called a \emph{central branch of $f^n$}. An iterate \( n \) is called a \emph{cutting time} if the image of the central branch of \( f^n \) contains \( c \).  

The cutting times are denoted by \( S_0, S_1, S_2, \ldots \), where \( S_0 = 1 \) and \( S_1 = 2 \). Note that the difference between two consecutive cutting times is again a cutting time, so we may define the following function
\[
Q \colon \mathbb{N} \to \mathbb{N} \cup \{0\},
\]
called the \emph{kneading map}, by
\[
S_k - S_{k-1} = S_{Q(k)}.
\] 
For the sake of completeness set $Q(0)=0$. If only finitely many cutting times exist, we set \(S_k=\infty\) and \(Q(k)=\infty\) for all sufficiently large \(k\).

If one is given a sequence of cutting times or a kneading map, then the associated unimodal map may be completely determined up to topological conjugacy.

\subsection{Hofbauer towers} The \emph{Hofbauer tower} is the disjoint union $\hat I=\bigsqcup_{n\in \N} D_n$ of open intervals $D_n\subset I$ where $D_1=(c,c_1)$ and inductively
\[
D_{n+1} = 
\begin{cases}
f(D_n), & \text{ if } c \notin \overline{D_n}, \\
(c_{n+1},c_1), & \text{ if } c \in \overline{D_n}.
\end{cases}
\]
Note that $c_n$ is always an endpoint of the interval $D_n$. Observe that the second case occurs precisely when \(n\) is a cutting time.
Moreover, for every $S_k<n\leq S_{k+1}$,
\[
D_n=(c_n,c_{n-S_k}).
\]

To define extended Hofbauer towers we need the following definition.
A forward iterate \( c_n \) is called a \emph{closest return} if 
\[
c_j \notin [c_n, \hat{c}_n] \quad \text{for} \quad 0 < j < n.
\]
The \emph{closest precritical points} is defined as follows:  
\[
z_0 := f^{-1}(c) \cap (0, c).
\]
Inductively,
\[
z_{k+1} := f^{-S_{k+1}}(c) \cap (z_k, c).
\]

The \emph{extended Hofbauer tower} is the disjoint union of copies of intervals $\widetilde{D}_n\subset [0,1]$ where $\widetilde{D}_1=(c,1)$  (for convenience we work with closures of $\widetilde D_{n}$ when needed) and
\[
\widetilde{D}_{n+1} = 
\begin{cases}
f(\widetilde{D}_n), & \text{ if } c \notin \widetilde{D}_n, \\
f(\widetilde{E}_n), & \text{ if } c \in \widetilde{D}_n.
\end{cases}
\]
where $\widetilde{E}_n$ is the component of $\widetilde{D}_n\setminus \{c\}$ containing $z_n$.
Note that it follows $D_n\subset \widetilde{D}_n$ for every $n\in\N$. Every cutting time for the original Hofbauer tower is also a cutting time for the extended tower, however there are also other cutting times. If $c\notin \widetilde{D}_n$, then $n$ is a \emph{co-cutting time} denoted by $\widetilde{S}_l$.

Let 
\[
s( n) =\max \{S_k| S_k<n\} 
\]
and
\[
\widetilde{s}( n) =\max \{\widetilde{S}_l| \widetilde{S}_l<n\}. 
\]

We use the following lemma \cite[Lemma 2]{Bruin94}:

\begin{lemma}
The following properties hold:
\begin{itemize}
\item A co-cutting time is never a cutting time.
\item $\widetilde{S}_0=\kappa$, where $\kappa>1$ is the smallest positive integer such that $c_{\kappa}>c$.
\item For $n\geq \widetilde{S}_0$, $\widetilde{D}_n=(c_{n-\widetilde{s}( n)},c_{n-s( n)})$. 
\item The difference between two subsequent co-cutting times is a cutting time. One can thus define a \emph{co-kneading map} $\widetilde{Q}:\N\to \N\cup \infty$ by
\[
S_{\widetilde{Q}(l)}=\widetilde{S}_l-\widetilde{S}_{l-1},
\]
using the same convention that $\widetilde{S}_l=\infty$ and $\widetilde{Q}(l)=\infty$ if $\widetilde{S}_l$ does not exist for some $l$.
\item Let $H_n\ni c_1$ be the maximal open interval such that $f^{n-1}|_{H_n}$ is monotone. Then 
\[f^{n-1}(H_n)=\widetilde{D}_n.\]
\item $\widetilde{D}_{S_k}=\left(c_{S_k-\widetilde{s}( S_k)}, c_{S_{Q(k)}}\right)$ and  $
\widetilde D_{\widetilde S_l}=\left(c_{\widetilde S_l-s(\widetilde S_l)},\,c_{S_{\widetilde Q(l)}}\right)$ for every $k,l\in \N$. 
\item Closest returns appear either at cutting or at co-cutting times.
\end{itemize}
\end{lemma}

\subsection{Recurrence}

Consider the unimodal map $f$ with infinite kneading sequence $\mathcal{K}(f)= \nu_1\nu_2\nu_3\cdots$.
The \(\omega\)-limit set of a point \( x \in [0,1] \) under \( f \) is defined by
\[
\omega(x,f) = \omega(x) = \{ y \in [0,1] \mid \exists\, n_1 < n_2 < \cdots \ \text{with} \ f^{n_i}(x) \to y \}.
\]

\begin{definition}
A point \( x \in [0,1] \) is called \emph{recurrent}, if for every open set \( U \ni x \), there exists an \( m \in \mathbb{N} \) such that \( f^{m}(x) \in U \); equivalently, \( x \) is recurrent if and only if \( x \in \omega(x) \). In terms of the kneading sequence, $c$ is recurrent if and only if every initial block $\nu_1\nu_2\cdots \nu_n$ of $\mathcal{K}(f)$ appears infinitely often in $\mathcal{K}(f)$. We shall say that point \(x\) \emph{infinitely recurrent} if $x$ is recurrent and non-periodic.
\end{definition}

\begin{definition}
    The point \( x \in [0,1] \) is \emph{uniformly recurrent} if for every open set \( U \ni x \), there exists an \( M \in \mathbb{N} \) such that whenever \( f^{j}(x) \in U \) for \( j \ge 0 \), then \( f^{j+k}(x) \in U \) for some \( 0 < k \le M \). The turning point $c$ is uniformly recurrent if and only if $\mathcal{K}(f)$ is a minimal sequence; that is, each initial block $\nu_1\nu_2\cdots\nu_n$ of $\mathcal{K}(f)$ appears in $\mathcal{K}(f)$ with bounded gap.
\end{definition}

\begin{definition}
We say \( x \in [0,1] \) is \emph{regularly recurrent} if for every open set \( U \ni x \), there exists an \( M \in \mathbb{N} \) such that \( f^{M \cdot j}(x) \in U \) for all \( j \in \mathbb{N} \). The turning point $c$ is regularly recurrent if and only if for each $n\in \N$ there exists $M\in \N$ such that $\nu_1\nu_2\cdots \nu_n = \nu_{M\cdot j+1}\nu_{M\cdot j+2}\cdots \nu_{M \cdot j+n}$ for all $j\in \N$; that is, $\mathcal{K}(f)$ is a \emph{Toeplitz sequence}.
\end{definition}

Every regularly recurrent point is uniformly recurrent, and every uniformly recurrent point is recurrent, but the converses do not hold in general.

\subsection{Homeomorphic restrictions for unimodal maps}

There is a long history of studying unimodal maps \(f\) for which \(f|_{\omega(c)}\) is a minimal homeomorphism of a Cantor set. One important case occurs when \(f|_{\omega(c)}\) is conjugate to an odometer; in this case we say that \(f\) has an embedded odometer.

\begin{comment}
\begin{definition}
    Let $\alpha = (j_1,j_2,\ldots)$ be a sequence of integers with each $j_i\geq 2$. Define $\Delta_\alpha$ to be the set of all sequences $(a_1,a_2,\ldots)$ such that $0\leq a_i\leq j_i-1$ for each $i$. Addition on $\Delta_\alpha$ is defined as follows: $$(x_1,x_2,\ldots) + (y_1,y_2,\ldots) = (z_1,z_2,\ldots)$$ where $z_1 = (x_1+y_1)\mod j_1$, and for each $i\geq 2$, $z_i = (x_i+y_i+r_{i-1})\mod j_i$ with $r_{i-1} = 0$ if $x_{i-1}+y_{i-1}+r_{i-2} < j_{i-1}$ and $r_{i-1} = 1$ otherwise (set $r_0 = 0$). Define $f_\alpha:\Delta_\alpha\to\Delta_\alpha$ by $$f_\alpha(x_1,x_2,\ldots) = (x_1, x_2, x_3,\ldots) + (1,0,0,\ldots).$$ The dynamical system $f_\alpha: \Delta_\alpha\to\Delta_\alpha$ is called the \emph{$\alpha$-adic odometer} or \emph{$\alpha$-adic adding machine}.
\end{definition}
\end{comment}

\begin{definition}
Let \(\alpha=(j_1,j_2,\ldots)\) be a sequence of integers with \(j_i\ge 2\) for all \(i\). Define
\[
\Delta_\alpha:=\prod_{i=1}^{\infty}\{0,1,\ldots,j_i-1\},
\]
the set of all sequences \(x=(x_1,x_2,\ldots)\) with \(0\le x_i\le j_i-1\).

Addition on \(\Delta_\alpha\) is defined by carrying. For \(x=(x_1,x_2,\ldots)\) and \(y=(y_1,y_2,\ldots)\), let
\(
x+y=(z_1,z_2,\ldots),
\)
where
\(
z_1\equiv x_1+y_1 \pmod{j_1},
\)
and for each \(i\ge2\),
\(
z_i\equiv x_i+y_i+r_{i-1}\pmod{j_i},
\)
with carry terms \(r_0:=0\) and
\[
r_{i-1}=
\begin{cases}
0, & \text{if } x_{i-1}+y_{i-1}+r_{i-2}<j_{i-1},\\
1, & \text{otherwise}.
\end{cases}
\]
Define \(f_\alpha:\Delta_\alpha\to\Delta_\alpha\) by
\[
f_\alpha(x)=x+(1,0,0,\ldots).
\]
The dynamical system \(f_\alpha:\Delta_\alpha\to\Delta_\alpha\) is called the \emph{\(\alpha\)-adic odometer}, or \emph{\(\alpha\)-adic adding machine}.
\end{definition}

It is well-known that if $f$ is infinitely renormalizable, then $f|_{\omega(c)}$ is conjugate to an odometer \cite[Proposition III.4.5]{deMelo_vanStrien}. Additionally, it is shown in \cite{BlockKeeslingMisiurewicz} that there exist non-renormalizable unimodal maps with embedded odometers, which are referred to as \emph{strange adding machines}. Moreover, these strange adding machines occur densely in the parameter space $[\sqrt{2},2]$ in the symmetric tent family. A construction to generate non-renormalizable unimodal maps with embedded odometers is defined in \cite{Alvin_strange_star} and serves as the basis for the construction provided in Section~\ref{Section: Absorbing Family}.

Not every unimodal map $f$ for which $f|_{\omega(c)}$ is a minimal homeomorphism has an embedded odometer. In \cite{Bruin_Homeomorphic_Restrictions}, a construction is provided to generate admissible kneading maps $Q(k)$ belonging to unimodal maps where $\omega(c)$ is a minimal Cantor set and $f|_{\omega(c)}$ is a homeomorphism. This construction produces many examples of unimodal maps where the kneading sequence $\mathcal{K}(f)$ is not a Toeplitz sequence, and thus $f|_{\omega(c)}$ is not conjugate to an odometer \cite{Alvin_ftcp}. At the time of publication of \cite{Bruin_Homeomorphic_Restrictions}, it was unknown whether there existed non-infinitely renormalizable unimodal maps $f$ with homeomorphic restrictions on their Cantor sets $\omega(c)$.

In \cite{Alvin_homeo}, a characterization of those unimodal maps $f$ for which $\omega(c)$ is a Cantor set and $f|_{\omega(c)}$ is a minimal homeomorphism is provided in terms of the kneading sequence. Given the minimal kneading sequence $\mathcal{K}(f)$, we define $X_f$ to be the one-sided shift space generated by $\mathcal{K}(f)$. 
\begin{theorem}\cite[Theorem 3.9]{Alvin_homeo}\label{Theorem:Alvin_homeo}
    Let \(f\) be a unimodal map. Then $f|_{\omega(c)}$ is a minimal homeomorphism on a Cantor set if and only if $\mathcal{K}(f)$ is an infinite aperiodic minimal sequence such that whenever $a0b,a'1b\in X_f$ where $|a|=|a'| < \infty$ and $b\in \{0,1\}^\N$, then $a=a'$ and $b=\mathcal{K}(f)$.
\end{theorem}

Equivalently, if \(f|_{\omega(c)}\) is a minimal homeomorphism of a Cantor set, then \(\mathcal K(f)\) is the unique point of \(X_f\) with two preimages, and any two \(n\)-th preimages of \(\mathcal K(f)\) differ in exactly one coordinate.

 \section{Characterizing Persistent Recurrence}

 The following definition can be found in \cite{GJ}.

 \begin{definition}
     We say $f$ is {\em critically monotonic} if either $c$ is not recurrent or there exists $k\in \N$ and a symmetric interval $U\ni c$ such that we can find a sequence of intervals $\{I_i\}$ and a sequence of increasing integers $\{s_i\}$ such that $c_k\in I_i$ for all $i\in\N$, $\cap I_i = \{c_k\}$, and $f^{s_i}: I_i\to U$ is one-to-one and onto.
 \end{definition}
 \begin{remark}
     Equivalently, $f$ is {\em not critically monotonic} if $c$ is recurrent and for every symmetric neighborhood $U\ni c$ there is an $N\in \N$ such that for all $n\geq N$, either $c_n\notin U$ or $\widetilde{D}_n \not\supset U$.
 \end{remark}

 \begin{definition}
     Let $\overline{x} = (\cdots, x_{-1},x_0)$ be a backward orbit under the unimodal map $f$.  Let $J \subset I$ be an interval. The sequence $(J_n)_{n\in \N_0}$ of intervals is called a {\em pull-back of $J$} along $\overline{x}$ if $J= J_0$, $x_{-k}\in J_k$ and $J_{k+1}$ is the largest interval such that $f(J_{k+1}) \subseteq J_k$ for all $k\in \N_0$. A pull-back is {\em monotone} if $c\notin J_n^\circ$ for every $n\in \N$. We can similarly define monotone pull-backs along finite backward orbits.
 \end{definition}

 The following definition by Blokh and Lyubich \cite{Blokh_Lyubich_Measurable_Dynamics_of_S_unimodal} is the interpretation of Yoccoz' condition $\tau(n)\to \infty$ from the complex setting to the real setting. The first known use of the term {\em persistent recurrence} seems to be in \cite{Lyubich_Milnors_attractors_persistent_recurrence}. 

 \begin{definition}
    The recurrent critical point $c$ is {\em reluctantly recurrent} if there exists a $\delta>0$ such that for every $l\in \N$ there is a backward orbit $\overline{x} = (x_{-l},\ldots, x_{-2},x_{-1},x)$ in $\omega(c)$ such that $B(x,\delta)$ has a monotonic pull-back along $\overline{x}$. Otherwise we say that $c$ is {\em persistently recurrent}.
 \end{definition}

In order to state Definition~\ref{Defintion: conditions C}, we need a few other facts.

\begin{definition}
Assume $f:[0,1]\to [0,1]$ is a unimodal map with the critical point $c$ and let $x\in [0,1]$.
If $\hat x=2c-x\in [0,1]$ we call $\hat x$ the {\em symmetric point of $x$}. 
A point $x$ is said to be {\em nice} if its forward orbit does not enter the interval $(x,\hat x)\subset [0,1]$. An interval $J\subset I$ is called {\em nice} if $J=(y,\hat y)$ for some nice point $y\in [0,1]$.
\end{definition}

%\red{It feels weird to have the absolute value and distance definition here; do we need to explicitly define them? Or is there a better place for it to be defined? \Jernej{I think it is a good place here}}
 %Let \( |\cdot| \) denote the absolute value. 
% The notation \( d(A,B) \) denotes the distance between sets or points \( A \) and \( B \).

\begin{definition}
Suppose that \( H_n(x) = (a,b) \ni x \) is the maximal interval on which \( f^n \) is diffeomorphic. Define
    \[
    r_n(x) := \min\{ |f^n(x) - f^n(a)|,\ |f^n(x) - f^n(b)| \}
    \]
    and
    \[
    R_n(x) := \max\{ |f^n(x) - f^n(a)|,\ |f^n(x) - f^n(b)| \}.
    \] 
\end{definition}

The following definition appears in \cite{Bruin_Absorbing_Cantor}.
\begin{definition}\label{Defintion: conditions C} Let $c$ be recurrent.
\begin{enumerate}
\item[(C1)] $f$ is Fibonacci-like, i.e.\ $k - Q(k)$ is bounded.
\item[(C2)] $Q(k) \to \infty$.
\item[(C3)] \(R_n(c_1)\to0\).
\item[(C4)] For every neighbourhood $U\ni c$, there exists $N\in \N$ such that for every $n>N$,
$\widetilde{D}_n$ does not cover either component of $U\setminus\{c\}$.
\item[(C5)] For every symmetric neighbourhood $U\ni c$, there exists $N\in \N$ such that for
every $n>N$, $U\not\subset \widetilde{D}_n$.
\item[(C6)] For every symmetric neighbourhood $U\ni c$, there exists $N\in\N$ such that for
every $n>N$,  $c_{n}\notin U$ or $\widetilde{D}_n\not\supset U$ (i.e. $f$ is not critically monotonic).
\item[(C7)] For every nice interval $U\ni c$,  there exists $N\in \N$ such that for every $n> N$,
$c_{n}\notin U$ or $\widetilde{D}_n\not\supset U$.
\item[(C8)] $r_n(c_1)\to 0$ ($\iff$ $c$ persistently recurrent).
\item[(C9)] $f|_{\omega(c)}$ is minimal and $\omega(c)$ Cantor set ($\iff$ $c$ uniformly recurrent \cite{Alvin_uniformly_recurrent}).
\end{enumerate}
\end{definition}

It is shown in Proposition~3.1 and Lemma~3.2 of \cite{Bruin_Absorbing_Cantor} that the following implications hold.
\begin{equation}
(C1) \Rightarrow (C2) \Leftrightarrow (C3) \Leftrightarrow (C4)\Rightarrow (C5)\Rightarrow (C6) \Rightarrow (C7) \Leftrightarrow (C8) \Rightarrow (C9)
\end{equation}

In \cite{Bruin_Absorbing_Cantor}, examples are provided to demonstrate that (C6)$\centernot\implies$ (C5) and (C7) $\centernot\implies$ (C6). Upon investigating these examples, it was determined that neither example yields an admissible kneading map nor a shift-maximal kneading sequence. We will explain why each example fails to correspond to a unimodal map and then provide modified examples to demonstrate that the claims made in \cite{Bruin_Absorbing_Cantor} indeed hold. We first comment on another result relating to persistent recurrence. 

In \cite{FPEP}, the following claim was made. Upon inspection, if this result were true, then critical monotonicity plus recurrence would imply reluctant recurrence (and thus persistent recurrence would imply that the turning point is not critically monotonic). Because we have examples of unimodal maps with persistently recurrent turning points that are critically monotonic (see Example~\ref{Example: Fixed C7 not C6}), this statement must be incorrect.

\begin{statement}\cite[Lemma 4.12]{FPEP}\label{Lemma: Sufficient for reluctant recurrence}
     Let $y\in \omega(c)$ and let $U$ be an interval such that $y\in U^\circ \subset U \subset [0,1]$. Assume for all $i$ the set $U$ can be pulled back monotonically along $(c_1,c_2,\ldots, c_{n_i})$ where $c_{n_i}\in U$ and $c_{n_i}\neq y$. Then $U$ can be monotonically pulled back along some infinite orbit $\overline{y} = (\ldots, y_{-2}, y_{-1},y)$ in $\omega(c)$. That is, $c$ is reluctantly recurrent.
 \end{statement}

 Observe that the above statement does not account for the possibility that the only accumulation points of the collection $\{c_{n_i}\}$ might lie in $\partial U$, which can happen, for example, if $U = (z_k,\hat{z}_k)$ as in Example~\ref{Example: Fixed C7 not C6}. In this case, we would not be able to find a neighborhood contained in $U$ with an infinite monotone pull-back along a folding point.

 We note that the previous statement was used in the proof of \cite[Theorem 4.13]{FPEP}, which states that $c$ is persistently recurrent if and only if the collection $\mathcal{E}$ of endpoints in the associated core inverse limit space is precisely the collection $\mathcal{F}$ of folding points.  We now provide a corrected proof of the theorem.

 First, we need some definitions.  We will work with 
the \emph{family of tent maps} $T_s(x) := \min\{ sx , s(1-x)\}$, $s \in (1,2]$, $x\in I$. 
In the next theorem we will use $T$ to denote the tent map.
We denote the \emph{critical point} by $c = \frac12$ and write $c_k := T^k(c)$.
With our choice of parameters, $c_2 < c < c_1$ and the interval $[c_2, c_1]$ 
is $T$-invariant.

 The \emph{inverse limit space}
$$
X :=\invlim{\{[0,1],T\}}=\{(\ldots, x_{-2}, x_{-1}, x_0): T(x_{-i})=x_{-(i-1)}, i\in\N\}
$$
is the collection of all backward orbits, equipped with the \emph{product metric}
$d(x,y) := \sum_{i \leq 0} 2^i |x_i-y_i|$. 
Denote by $\pi_{i}: X \to I$, $\pi_{i}(x) := x_{-i}$, 
the \emph{coordinate projections} for $i\in\N_0 := \N \cup \{ 0 \}$.
The \emph{shift homeomorphism} $\sigma:X \to X$ is defined by

\begin{equation}\label{def:shift}
\sigma(\ldots, x_{-2}, x_{-1}, x_0) := (\ldots, x_{-1}, x_0, T(x_0)).
\end{equation}

In the following theorem we will restrict our attention to the (indecomposable) inverse limit space $\invlim{\{[c_2,c_1], T|_{[c_2,c_1]}\}}$.
 {\em Basic arcs} are maximal closed connected sets $A \subset \invlim{\{[c_2,c_1], T|_{[c_2,c_1]}\}}$ on which $\pi_0:A \to I$
is injective, where $\pi_0(x) = x_0$ is the projection on the zero-th coordinate
of $x \in \invlim{\{[c_2,c_1], T|_{[c_2,c_1]}\}}$. The inverse limit space $\invlim{\{[c_2,c_1], T|_{[c_2,c_1]}\}}$ can then be represented as the union of its basic arcs, glued together
in an intricate way (see Lemma 2 in \cite{Br1}).

Although Lemma~4.12 of \cite{FPEP} is not valid as stated, the endpoint characterization in \cite[Theorem~4.13]{FPEP} remains true; we provide a corrected proof below. For $\epsilon>0$ and $x\in [0,1]$ let $B(x,\epsilon):=(x-\epsilon,x+\epsilon)\cap [0,1]$.

 \begin{theorem}\cite[Theorem 4.13]{FPEP}\label{Theorem: fixed FPEP endpoint characterization}
     For $\invlim{\{[c_2,c_1], T|_{[c_2,c_1]}\}}$ it holds that $\mathcal{F}= \mathcal{E}$ if and only if $c$ is persistently recurrent.
 \end{theorem}
 \begin{proof} 
 %\Jernej{I am in favor of either extending the second sentence in the following paragraph or using the teal version below this paragraph.}
% Suppose that $c$ is reluctantly recurrent. Then there exists $\delta > 0$ such that for infinitely many $n_i$, $\widetilde{D}_{n_i} \supset (c_{n_i}-\delta, c_{n_i}+\delta)$.
 %Let $y_0$ be an accumulation point of $\{c_{n_i}\}$. Without loss of generality, we can assume $d(y_0, c_{n_i}) < \delta/2$ for all $n_i$ by passing to an appropriate subsequence.  Hence $y_0\in B(c_{n_i},\delta) \subset \widetilde{D}_{n_i}$ for all $i$.  Therefore, there exists $\overline{y}=(\ldots, y_{-2},y_{-1},y_0)\in \mathcal{F}$ such that $I_{n_i-1}(\overline{y})\cdots I_{1}(\overline{y}) = e_1\cdots e_{n_i -1}$ for infinitely many $i$. Observe $B(y_0,\delta/2) \subset \widetilde{D}_{n_i}$ for all $i$. Let $J = B(y_0,\delta/2)$. Then there is an infinite monotone pullback of $J$ along $\overline{y}$. Note that $\invlim\{J_n,T|_{J_n}\}$ is an arc in $\invlim{\{[c_2,c_1], T|_{[c_2,c_1]}\}}$ that contains $\overline{y}$ in its interior, and thus $\overline{y}$ is not an endpoint.
 Suppose that \(c\) is reluctantly recurrent. Then there exists \(\delta>0\) such that for every \(l\in\mathbb N\) there is a backward orbit
\(
(x^{(l)}_{-l},x^{(l)}_{-(l-1)},\ldots,x^{(l)}_{-1},x^{(l)}_0)\subset \omega(c)
\)
for which \(B(x^{(l)}_0,\delta)\) admits a monotone pull-back of length \(l\). Since \(\omega(c)\) is compact, by passing to a subsequence and using a diagonal argument, we may assume that for each fixed \(m\ge0\),
\(
x^{(l)}_{-m}\to x_{-m} \text{ as }  l\to\infty.
\)
By continuity, this defines an infinite backward orbit
\[
\overline{x}=(\ldots,x_{-2},x_{-1},x_0)\in \invlim{\{[c_2,c_1], T|_{[c_2,c_1]}\}}
\]
such that \(x_{-m}\in\omega(c)\) for every \(m\ge0\).

Fix \(m\ge0\). For all sufficiently large \(l\), we have
\(
B(x_0,\delta/2)\subset B(x^{(l)}_0,\delta).
\)
Hence \(B(x_0,\delta/2)\) admits a monotone pull-back of length \(m\) along
\(
(x^{(l)}_{-m},\ldots,x^{(l)}_0).
\)
By compactness of the corresponding pull-back intervals, and after possibly shrinking the interval to \(B(x_0,\delta/3)\), passing to the limit gives a monotone pull-back of length \(m\) along
\((x_{-m},\ldots,x_0).\)
Since \(m\) was arbitrary, \(B(x_0,\delta/3)\) admits an infinite monotone pull-back \((J_n)_{n\in\mathbb N_0}\)
along \(\overline{x}\). Therefore \(\invlim\{J_n,T|_{J_n}\}\)
is an arc in \(\invlim\{[c_2,c_1],\,T|_{[c_2,c_1]}\}\) containing \(\overline{x}\) in its interior. Since \(x_{-m}\in\omega(c)\) for every \(m\ge0\), we have \(\overline{x}\in\mathcal F\). Thus \(\overline{x}\notin\mathcal E\), and hence \(\mathcal F\neq\mathcal E\).

 For the other direction, the proof is analogous to the second paragraph of the proof of \cite[Theorem 4.13]{FPEP}. Nevertheless, we give it here for the sake of completeness.\\
 Let $c$ be persistently recurrent and assume that $x=(\ldots, x_{-1}, x_0)\in \invlim{\{[c_2,c_1], T|_{[c_2,c_1]}\}}$ such that $x\in \mathcal{F}\setminus\mathcal{E}$. Replacing \(x\) by \(\sigma^{-j}(x)\) for some sufficiently large \(j\in\mathbb N\), if necessary, we may assume that \(x\) lies in the interior of a basic arc. Let $A$ be a subset of the basic arc of $x$ such that $\partial A\cap \mathrm{Orb}(c)=\emptyset$ 
and such that $x\in A^\circ$. Let $A_k:=\pi_k(A) \subseteq[c_2, c_1]$ 
for every $k\in\N_{0}$. Denote by $J:=A_0$ and by $(J_n)_{n\in\N_0}$ the pullback of $J$ along $x$. 
Note that $A_n\subset J_n$ for every $n\in\N_0$. Since $c$ is persistently recurrent, 
there exists the smallest $N\in\N$ such that $c\in J^\circ_{N}$. Thus $A_0=J_0,A_1=J_1,\ldots, A_{N-1}=J_{N-1}$ but $A_N\subsetneq J_N$. 
Since $c\not\in A_n^\circ$ for 
every $n\in\N$ (because otherwise $\partial A\cap \mathrm{Orb}(c)\neq\emptyset$), it follows that $c$ is an endpoint of $A_{N}$, since $T(c)=c_1\in \partial([c_2,c_1])$. But then $c_N$ is an endpoint of $A_0=A$, 
which is a contradiction.
 \end{proof}

We are now ready to investigate the claim regarding (C6)$\centernot\implies$ (C5).

\begin{example}\cite[Example 11.3]{Bruin_Absorbing_Cantor}\label{Example: Bad C6 not C5}
    Let $Q(0) = Q(1) = 0$ and for $k>1$, define $$Q(k) = \begin{cases}
        0 & \text{ if } k\equiv 0 \mod 4 \text{ or } k\equiv 3 \mod 4,\\
        k-3 & \text{ if } k \equiv 1 \mod 4, \\
        k-1 & \text{ if } k \equiv 2 \mod 4.
    \end{cases}$$

    One can always take an admissible kneading map and convert it into a shift-maximal kneading sequence. In this case, if one were to create the associated symbolic sequence, one would obtain $$e = 1011001010 \cdots$$ Observe that $\sigma^3(e) \succeq e$ in the parity lexicographical ordering, and thus $e$ is not shift-maximal. Hence this inadmissible kneading map does not correspond to a unimodal map. 
\end{example}

It is possible to modify the previous example so that the kneading map is admissible. Further, the methods provided in \cite{Bruin_Absorbing_Cantor} to justify that (C6) and (C5) are distinct properties will hold. 

\begin{example}[Modified Example~\ref{Example: Bad C6 not C5}]
    Let $Q(0) = Q(1)=Q(2) = Q(4) = 0$, $Q(3) = Q(6) = 1$, $Q(5) = 2$, and for $k>6$ define $$Q(k) = \begin{cases}
        0 & \text{ if } k\equiv 0 \mod 4 \text{ or } k\equiv 3 \mod 4,\\
        k-3 & \text{ if } k \equiv 1 \mod 4, \\
        k-1 & \text{ if } k \equiv 2 \mod 4.
    \end{cases}$$
\end{example}

The following example was presented in \cite{Bruin_Absorbing_Cantor} to demonstrate (C7)$\centernot\implies$(C6). We demonstrate that the resulting kneading map fails to be admissible.
\begin{example}\cite[Example 11.1]{Bruin_Absorbing_Cantor}
    Let $Q(0) = Q(1) = Q(2) = 0$, $k_0=1$, and $k_1 = 3$. Then we define inductively,
    \begin{align*}
        Q(k_i) & = k_{i-1},\\
        Q(k_i+1) & = 0, \\
        Q(k_i+2) &= k_{i-1}+2,\\
        Q(k_i+3) &=k_i+2,\\
        Q(k_i+4) &=k_i,\\
        Q(k_i+5) & = 1,\\
        Q(k_i+5+j) &= Q(j) \text{ for } j=1,2,\ldots, k_{i-1}+3, \\
        k_{i+1} &= k_i + k_{i-1}+9
    \end{align*}

    Recall that admissibility in terms of the kneading map is guaranteed by $$\{Q(k+j)\}_{j\geq 1} \succeq \{ Q(Q^2(k) + j)\}_{j\geq 1},$$ where $\succeq$ is the lexicographical ordering.
    We show that the admissibility condition is not met for this example. Observe that $Q(Q^2(8)+j) = Q(j)$ for all $j$. It is easy to check that $Q(8+j) = Q(j)$ for $j=1, \ldots, 5$. However, when $j=6$ we have $Q(8+6) = 0$ and $Q(Q^2(8)+6) = 5$. This implies that $\{Q(Q^2(8) + j)\}_{j\geq 1} \succeq \{Q(8+j)\}_{j\geq 1}$. Hence there is no unimodal map with the kneading map described in \cite[Example 11.1]{Bruin_Absorbing_Cantor}. 
\end{example}

We now provide an example of a unimodal map for which there exists a symmetric interval $U$ such that for infinitely many $n\in \N$, $c_n\in U$ and $\widetilde{D}_n\supset U$, however (C7) is still satisfied. This example first appeared in \cite{Alvin_Brucks_AMKME} and provides an example of a unimodal map with embedded odometer for which $Q(k)\not\to\infty$, however the collection of folding points for the core inverse limit space is precisely the collection of endpoints (and thus $c$ is persistently recurrent).

\begin{example}\cite[Example 3.10]{Alvin_Brucks_AMKME}\label{Example: Fixed C7 not C6}
    We define the kneading map of a unimodal map with embedded odometer given by $\alpha = (25,5,5,5,\ldots)$. 
    Let 
    
    \[Q(k) = \begin{cases}
	0 & \text{if}\ k=1,2 ,\\
	1 & \text{if}\ k = 3+2\cdot i_2 + \sum_{n\geq 3} n\cdot i_n, \ \text{where}\ i_2 \in \{0,1,2\},\\ & i_n \in \{0,1,2,3\} \text{ for } n\geq 3, i_3 \neq 0 \text{ gives } i_2 = 2, \text{and} \\ & i_n \neq 0 \text{ for } n\geq 4 \text{ gives } i_{n-j}=3 \text{ for } 1\leq j\leq n-3,\\
	2 & \text{if}\ k = 4+2\cdot i_2 + \sum_{n\geq 3} n\cdot i_n, \ \text{where}\ i_2 \in \{0,1\},\\ & i_n \in \{0,1,2,3\} \text{ for } n\geq 3, i_3\neq 0 \text{ gives } i_2 = 1, \text{and} \\& i_n \neq 0 \text{ for } n\geq 4 \text{ gives } i_{n-j}=3 \text{ for } 1\leq j\leq n-3,\\
	3=4-1 & \text{if}\ k=8=4+2\cdot 2,\\
	8 & \text{if}\ k = 11+3\cdot i_3 + \sum_{n\geq 4} n\cdot i_n, \ \text{where}\ i_3 \in \{0,1\},\\ & i_n \in \{0,1,2,3\} \text{ for } n\geq 4, i_4\neq 0 \text{ gives } i_3 = 1, \text{and} \\ & i_n\neq 0 \text{ for } n\geq 5 \text{ gives } i_{n-j}=3 \text{ for } 1\leq j \leq n-4,\\
	10 = 11-1 & \text{if}\ k = 17 = 11+3\cdot 2,\\
	17 & \text{if}\ k = 21+4\cdot i_4 + \sum_{n\geq 5} n\cdot i_n, \ \text{where}\ i_4 \in \{0,1\},\\ & i_n \in \{0,1,2,3\} \text{ for } n\geq 5, i_5\neq 0 \text{ gives } i_4 = 1, \text{and} \\ & i_n\neq 0 \text{ for } n\geq 6 \text{ gives } i_{n-j}=3 \text{ for } 1\leq j \leq n-5,\\
	20=21-1 & \text{if}\ k=29 = 21+4\cdot 2,\\
	29 & \text{if}\ k = 34+5\cdot i_5 + \sum_{n\geq 6} n\cdot i_n, \ \text{where}\ i_5 \in \{0,1\},\\ & i_n \in \{0,1,2,3\} \text{ for } n\geq 6, i_6\neq 0 \text{ gives } i_5 = 1, \text{and} \\ & i_n\neq 0 \text{ for } n\geq 7 \text{ gives } i_{n-j}=3 \text{ for } 1 \leq j\leq n-6,\\
	33=34-1 & \text{if}\ k = 44 = 34+5\cdot 2,\\
	44 & \text{if}\ k = 50+6\cdot i_6 + \sum_{n\geq 7} n\cdot i_n, \ \text{where}\ i_6 \in \{0,1\},\\ & i_n \in \{0,1,2,3\} \text{ for } n\geq 7, i_7\neq 0 \text{ gives } i_6 = 1, \text{and} \\ & i_n\neq 0 \text{ for } n\geq 8 \text{ gives } i_{n-j}=3 \text{ for } 1\leq j\leq n-7,\\
	etc. & 
\end{cases}\]  

The following table is provided in \cite{Alvin_Brucks_AMKME} to aid in better understanding how $Q(k)$ is defined, so we include it here.

\begin{table}[h!] 
\begin{tabular}{l | c c c c c c c c c c c c c c c c c c c c c c c}
$k$ & 0 & 1 & 2 & 3 & 4 & 5 & 6 & 7 & 8 & 9 & 10 & 11 & 12 & 13 & 14 & 15 & 16 \\ \hline
$Q(k)$ & \ & 0 & 0 & 1 & 2 & 1 & 2 & 1 & 3 & 2 & 1 & 8 & 2 & 1 & 8 & 2 & 1\\ 
$S_k$ & 1 & 2 & 3 & 5 & 8 & 10 & 13 & 15 & 20 & 23 & 25 & 45 & 48 & 50 & 70 & 73 & 75  
\end{tabular}

\vspace{.1in}

\begin{tabular}{l | c c c c c c c c c c c c c c c c}
$k$ &  17 & 18 & 19 & 20 & 21 & 22 & 23 & 24 & 25 & 26 & 27 & 28 & 29 \\ \hline
$Q(k)$ & 10 & 8 & 2 & 1 & 17 & 8 & 2 & 1 & 17 & 8  & 2 & 1 & 20\\ 
$S_k$ & 100 & 120 & 123 & 125 & 225 & 245 & 248 & 250 & 350 & 370 & 373 & 375 & 500
\end{tabular}

\vspace{.1in}

\begin{tabular}{l | c c c c c c c c c c c }
$k$  & 30 & 31 & 32 & 33 & 34 & 35 & 36 & 37 & 38 & 39 & $\cdots$ \\ \hline
$Q(k)$ & 17 & 8 & 2 & 1 & 29 & 17 & 8 & 2 & 1 & 29  & $\cdots$ \\ 
$S_k$  & 600 & 620 & 623 & 625 & 1125 & 1225 & 1245 & 1248 & 1250  & 1750  & $\cdots$
\end{tabular}
\end{table}

Note that for all $n\geq 2$ we have that $4\cdot 5^n-2$ is a co-cutting return time and $4\cdot 5^n$ is a cutting return time. Note that $3\cdot 5^n$ is the last cutting time prior to $4\cdot 5^n$ while $4\cdot 5^{n}-5$ is the last co-cutting time prior to $4\cdot 5^n-2$. Hence $\widetilde{D}_{4\cdot 5^n-2} = [c_3, c_{3\cdot 5^n-2}]$ for all $n\geq 2$. Let $U = (z_1,\hat{z}_1)$. Observe that $U\subset \widetilde{D}_{4\cdot 5^n-2}$ for all $n\geq 2$ with $c_3 < z_1< c < c_{4\cdot 5^n-2} < \hat{z}_1 < c_{3\cdot 5^n-2}$. In other words, (C6) fails to hold.
We note that it was first shown in \cite{Alvin_Brucks_AMKME} that the set of endpoints for the core inverse limit associated with this map is precisely the set of folding points, and thus $c$ is persistently recurrent. That is, (C7) holds. Further note that $U = (z_1,\hat{z}_1)$ is the only symmetric neighborhood such that for infinitely many return times $n$, $c_n\in U$ and $\widetilde{D}_n\supset U$.

\end{example}

In the following proposition we will use the fact that, by construction of the extended Hofbauer towers, pull-back branches are cut precisely at closest precritical points \(z_k\) and their symmetric points \(\hat z_k\).

\begin{proposition}
    Suppose $c$ is persistently recurrent and there exists a symmetric neighborhood $U$ such that for infinitely many return times $n$, $c_n\in U$ and $\widetilde{D}_n\supset U$. Then $U = (z_k,\hat{z}_k)$ for some $k\in \N$.
\end{proposition}
\begin{proof}
Let \(U=(u,\hat u)\), where \(u<c<\hat u\). Let \(\{n_i\}\) be a sequence of
return times such that
\(c_{n_i}\in U\) and \(D_{n_i}\supset U\)
for every \(i\). Since \(c\) is persistently recurrent, condition \((C8)\) holds, and hence \(r_n(c_1)\to0 \).
It follows that, for all sufficiently large \(i\), the branch of the extended
Hofbauer tower corresponding to \(\widetilde D_{n_i}\) is cut at the boundary of
\(U\). By the observation before this proposition, this boundary must be a closest precritical
pair. Hence there exists \(k\in\mathbb N\) such that
\(\partial U=\{z_k,\hat z_k\}\).
Therefore \(U=(z_k,\hat z_k)\).
\end{proof}

% \begin{proof}
%     Suppose $c$ is persistently recurrent and let $U$ be a symmetric neighborhood such that for a sequence $\{n_i\}$ of return times, $c_{n_i}\in U$ and $\widetilde{D}_{n_i}\supset U$ for all $i$. Without loss of generality, suppose that each $n_i$ is a co-cutting return time. Note that for each $n_i$, there exists a $k_i$ such that either $c_{n_i-s(n_i)} < z_{k_i} < c_{n_i}$ or $c_{n_i} < \hat{z}_{k_i} < c_{n_i-s(n_i)}$; if no such $k_i$ existed, then $n_i$ would not be a co-cutting return time. Let $k$ be the smallest integer such that $z_k\in \overline{U}$, but $z_{k-1}\notin \overline{U}$. Fix $0 < \epsilon < \frac{1}{2}|z_{k+1}-z_k|$. As $c$ is persistently recurrent, there exists some $I$ such that for all $i\geq I$, $r_{n_i-1}(c_1) < \epsilon$. Observe that $k_i = k$ for all $i\geq I$ as otherwise we would have $c_{n_i}\notin U$ or $\widetilde{D}_{n_i}\centernot\supset U$. As $\epsilon$ can be taken arbitrarily small, it follows that both $c_{n_i}$ and $c_{n_i-s(n_i)}$ get arbitrarily close to the boundary of $U$, and thus $U = (z_k,\hat{z}_k)$.
% \end{proof}

%
%
%
%
%
%
%
%
%
\subsection{Symbolically characterizing persistent recurrence}\label{Section: Symbolic Characterization}

The following definitions are taken from \cite{Sands}. In what follows, we assume that $c$ is recurrent and non-periodic.
We define 
\[R(x) := \min\{j\geq 1: c\in f^j[c;x] \},\]  where $[y;z]:=[\min\{y,z\},\max\{y,z\}]\subset [0,1]$ and let
$\mathcal{R}(n) = R(c_n)$. 
For $i\geq 1$ let 
\[
e_i:=
\begin{cases}
0, &  f^i(c)\in[0, c),\\
%*, &  f^i(c)=c,\\
1, &  f^i(c)\in(c, 1].
\end{cases}
\]
Since \(c\) is non-periodic, \(c_i\neq c\) for all \(i\ge1\).
It follows that for each $n\geq 1$ we have that \[\mathcal{R}(n) = \min\{j\geq 1: e_{n+j}\neq e_j\},\] where 
\[\mathcal{K}(f) = e_1e_2e_3\cdots\] 
is the {\em kneading sequence} of the unimodal map $f$. Because $\mathcal{K}(f)$ is shift-maximal, we always have $\RR(n) = S_k$ for some $k\geq 0$. We define $\RR^-(n) = S_{k-1}$ if $k > 1$.

The sequences of cutting times, co-cutting times, and return times can be defined in terms of the function $\RR$ as follows. 
\begin{itemize}
    \item The sequence of cutting times is given by $S_1 = 1$ and $S_{i+1} = S_i + \RR(S_i)$ for all $i\geq 1$.
    \item The sequence of co-cutting times is given by $T_1 = \min\{j>1\mid e_j = 1\}$ and $T_{i+1} = T_i + \RR(T_i)$ for all $i\geq 1$.
    \item The first return time is $M_1 = T_1 -1$ and is a cutting time. The subsequent return times are given by $M_{i+1} = M_i + \RR^{-}(M_i)$ for all $i\geq 1$. Observe that the return times alternate between cutting times and co-cutting times. Further, each return time is either the last cutting time before a co-cutting time or the last co-cutting time before a cutting time.
\end{itemize}

We note that for any $n>1$, if $\RR(n) > \RR(i)$ for all $1\leq i < n$, then $n$ is a \emph{closest return time}. Every closest return time is necessarily a return time.  The following lemma relates the functions $\RR$ and $R$ with the distance a point in $[0,1]$ lies from the turning point $c$.

\begin{lemma}\label{Lemma: close to c}
    For each $K\in \N$, we may define $\epsilon = |z_K-c|$ such that $c_n\in B(c,\epsilon)$ if and only if $\RR(n) > S_K$.  More generally, for each $\epsilon > 0$, there exists an $M\in \N$ such that whenever $R(x)>M$, then $|x - c| < \epsilon$ and if $|x - c| < \epsilon$, then $R(x)\geq M$.
\end{lemma}

Recall the extended Hofbauer tower. For each $n\in \N$, we have that $\widetilde{D}_n = [c_{n-\widetilde{s}(n)};c_{n-s(n)}]$ is the image of the maximal interval of monotonicity containing $c_1$ under $f^{n-1}$.  Here $s(n)$ denotes the last cutting time prior to $n$ and $\widetilde{s}(n)$ denotes the last co-cutting time prior to $n$. If \(n\) is a cutting or co-cutting time, then \(c\in \widetilde D_n\).

\begin{lemma}\cite[Lemma 11.1]{Bruin_Absorbing_Cantor}\label{Lemma: K consecutive terms large}
    Suppose that $R_{n-1}(c_1) \geq \epsilon$. Then there exists $K = K(\epsilon)$ such that $\{n, n+1, \ldots, n+K\}$ contains a cutting or a co-cutting time. If $r_{n-1}(c_1) \geq \epsilon$, then $\{n, n+1,\ldots, n+K\}$ contains a cutting and a co-cutting time. In particular, this holds if $n$ is a cutting or a co-cutting time.
\end{lemma} We emphasize that it is possible for $r_{n-1}(c_1) < \epsilon$ and $\{n,n+1,\ldots, n+K\}$ to contain both a cutting and a co-cutting time.

We use the following notation throughout this section. If $M_i$ denotes a return time, then $N_i$ denotes the first cut or co-cut after that return time. We note that $N_i$ will be a co-cutting time whenever $M_i$ is a cutting time, and vice versa; that is, $N_i$ and $M_i$ are of opposite type. We let $L_i = s(M_i)$ when $M_i$ is a cutting time, and $L_i = \tilde{s}(M_i)$ when $M_i$ is a co-cutting time; that is, $L_i$ is the last cutting time or co-cutting time prior to $M_i$ that is of the same type as $M_i$. The following two lemmas are useful to keep track of different values for $\RR$ and the connections to the kneading map when considering return times $M_i$ and the associated cutting and co-cutting times $L_i$ and $N_i$; these results follow from \cite{Sands}.

\begin{lemma}\label{Lemma: helpful kneading facts}
    Let $\{M_i\}$ be the sequence of return times. Then there is a sequence $\{t_i\}$ such that $\RR(M_i) = S_{t_i}$ for all $i$. The following hold.
    \begin{enumerate}
        % \item \sout{For every \sout{(co-)cutting} \Jernej{cutting or co-cutting} time $n_j$ such that $M_i < n_1 = N_i < n_2 < \cdots < n_k=M_{i+1}$, it follows that $n_j = M_i + S_{t_i - (k+1-j)}$.}
        \item Let $M_i<n_1=N_i<n_2<\cdots<n_k=M_{i+1}$
        be the cutting and co-cutting times between $M_i$ and $M_{i+1}$. Then, for each $1\le j\le k$, $n_j=M_i+S_{\,t_i-(k+1-j)}$.
        \item If $M_{i+1} = N_i$, then $M_{i+1} - L_{i+1} = N_{i} - M_{i-1} = S_{t_{i-1}}$. 
        \item If $M_{i+1} > N_i$, then $M_{i+1} - L_{i+1} = S_{t_i-1}- S_{t_i-2}=S_{Q(t_i-1)}$.
        \item $N_{i+1} - M_{i+1} = (N_{i+1} - M_{i}) - (M_{i+1} - M_i) = S_{t_i} - S_{t_i-1} = S_{Q(t_i)}$.
    \end{enumerate}
\end{lemma}

\begin{lemma}\label{Lemma: helpful cutting facts}
    Let $M_i$ be a co-cutting return time. Then for all cutting times $M_i < n \leq M_{i+1}$, the following hold (symmetric results hold when $M_i$ is a cutting time and each $n$ is a co-cutting time):
    \begin{enumerate}
        \item If $n < M_{i+1}$, then $\RR(n-\tilde{s}(n))=\RR(n-M_i) = \RR(n)$.
        \item If $n=M_{i+1}$, then $\RR(M_{i+1}-\tilde{s}(M_{i+1}))=\RR(M_{i+1}-M_i) = N_{i+1}-M_{i+1}$.
        \item If $n = N_i$, then $\RR(N_i - {s}(N_i))=\RR(N_i - M_{i-1}) = \RR(S_{t_{i-1}}) = S_{Q(t_{i-1}+1)}$.
        \item If $N_i < n \leq M_{i+1}$, then $\RR(n - s(n)) \leq n-s(n) = \RR(s(n))$.
    \end{enumerate}
\end{lemma}

We now present a symbolic characterization of persistent recurrence before providing symbolic characterizations of (C5) and (C6); since both (C5) and (C6) imply persistent recurrence, their characterizations will provide sufficient conditions for a unimodal map to have a persistently recurrent turning point. 
%For completeness, we note that by \cite{Bruin_Absorbing_Cantor}, (C4) is equivalent to $Q(k)\to \infty$.

\begin{theorem}\label{Theorem: Main Symbolic Characterization}
    Given a unimodal map $f$ with kneading sequence $\mathcal{K}(f)$, $c$ is persistently recurrent if and only if for all $k\in \N$ there exists an $I\in \N$ such that for all $i\geq I$ one of the following holds:
    \begin{enumerate}
        \item\label{Main 1} $N_i - M_i \geq k$;
        \item\label{Main 2} $N_i - M_i < k$, but $\RR(M_i)\geq k$ and $\RR(M_i - L_i)\geq k$;
        \item\label{Main 3} $\RR(M_i) < k$, but $\RR(N_i)\geq k$ and $\RR(N_i - M_{i-1})\geq k$. 
        %\item whenever $\{i_j\}$ is a sequence of integers greater than $I$ such that $N_{i_j} - M_{i_J} < k$ where $\RR(M_{i_j})\geq k$, then $\RR(M_{i_j}-L_{i_j}) \geq k$;
        %\item whenever $\{i_j\}$ is a sequence of integers greater than $I$ such that $\RR(M_{i_j})\leq k$, then $\RR(N_{i_j})\geq k$ and $\RR(N_{i_j} - M_{i_{j}-1})\geq k$.
    \end{enumerate}
\end{theorem}
\begin{proof}
    Suppose first that $c$ is persistently recurrent. Fix $\epsilon > 0$ and note that there exists a $K = K(\epsilon)$ such that if $x\in B(c,\epsilon)$, then $R(x)\geq K$ and conversely, if $R(x)\geq K+1$, then $x\in B(c,\epsilon)$. By the contraction principle, there exists a $\delta > 0$ such that if $T$ is an interval with $|T|> \epsilon/2$, then $|f^n(T)|>\delta$. Let $J$ be chosen large such that $\delta/J < \epsilon/2$. Since $c$ is persistently recurrent, there exists an $N\in \N$ such that for all $n\geq N$, $r_{n-1}(c_1)< \delta/J$.
    
    Suppose that $n = M_i \geq N$ is a return time and without loss of generality, suppose $n$ is a cutting time. As $r_{n-1}(c_1) < \delta/J$, this implies that either $|c_n - c_{n-s(n)}| < \delta/J$ or $|c_n-c_{n-\tilde{s}(n)}| < \delta/J$. If $c_{n-\tilde{s}(n)}\in B(c,\epsilon)$, then $\RR(n),\RR(n-\tilde{s}(n)) \geq K$; in this case, $\RR(n-\tilde{s}(n)) = N_i - M_i \geq K$. 

    Now consider the case where $c_{n-\tilde{s}(n)}\notin B(c,\epsilon)$. Then $N_i - M_i \leq K$, and without loss of generality we will suppose $N_i - M_i < K$. If $|c_n - c_{n-s(n)}| < \delta/J$, then $\RR(n) \geq K$ and $\RR(n-s(n))\geq K$; that is $\RR(M_i)\geq K$ and $\RR(M_i -L_i)\geq K$. On the other hand, if $|c_n - c_{n-\tilde{s}(n)}| < \delta/J$, then there exists a $z_r \in (c_n; c_{n-\tilde{s}(n)})$ such that $\RR(n-\tilde{s}(n)) = S_r < K$. Thus $c\in f^{S_r}([c_n; c_{n-\tilde{s}(n)}])$. Since $|c_n - c| > \epsilon/2$, we have $|c_{n + S_r} - c_{S_r}| > \delta$. Hence $|c_{n+S_r}- c_{n-\tilde{s}(n)+S_r}|<\delta/J < \epsilon/2$. That is, $\RR(n+S_r) = \RR(N_i) \geq K$ and $\RR(n-\tilde{s}(n)+S_r) = \RR(N_i - M_{i-1})\geq K$.

    Conversely, suppose that for all $k\in \N$ there exists an $I$ such that for all $i\geq I$ one of the three conditions holds. %\Jernej{In what follows, are we using contraction principle correctly or should it rather be as I write?} \sout{Fix $\epsilon > 0$ and let $\delta > 0$ be such that if $|T|<\delta$, then $|f^{-n}(T)| < \epsilon/2$.}
    By the contraction principle, for every \(\epsilon>0\) there exists \(\delta>0\) such that if \(T\) is an interval and
    \(|f^n(T)|<\delta,\)
    then \(|T|<\epsilon/4\), provided \(f^n|_T\) is monotone.
    Fix $J$ large so that $\delta/J < \epsilon/2$. There exists a $K = K(\delta/J)$ such that if $x\in B(c,\delta/J)$, then $R(x) \geq K-1$ and if $R(x)\geq K$, then $x\in B(c,\delta/J)$.  Let $I$ be chosen such that our conditions hold for this $K$ value whenever $i\geq I$. 

    First suppose that $N_i - M_i \geq K$. Then both $c_{M_i}$ and $c_{M_i-M_{i-1}}$ lie inside $B(c,\delta/J)$, and thus $|c_{M_i} - c_{M_i - M_{i-1}}| < 2\delta/J < \epsilon$. Further, by the contraction principle, it follows that $|c_{n}- c_{n-M_{i-1}}|< \epsilon/2$ for all $M_{i-1} < n \leq M_i$; that is, $r_{n-1}(c_1) < \epsilon$ for all $M_{i-1} < n \leq M_{i}$.

    % \Jernej{Instead of the following paragraph I suggest something like the paragraph below in teal.}
    
    % Now suppose that $\RR(M_i) < K$, but $\RR(N_i)\geq K$ and $\RR(N_i-M_{i-1})\geq K$. Let $N_i - M_i = S_r$ where $z_r\in (c_{M_i - M_{i-1}}; c_{M_i})$. Then $c\in f^{S_r}([c_{M_i - M_{i-1}}; c_{M_i}]) = [c_{N_i - M_{i-1}};c_{N_i}]$. Because both $c_{N_i - M_{i-1}},c_{N_i}\in B(c,\delta/J)$, it follows that $|c_{N_i - M_{i-1}} - c_{N_i}|< 2\delta/J$, and hence by the contraction principle, $|c_{M_i- M_{i-1}} - c_{M_i}| < \epsilon/2$. By additionally applying the contraction principle, we see that $r_{n-1}(c_1) < \epsilon$ for all $M_{i-1} < n \leq M_i$. 

    Now suppose that \(\RR(M_i)<K\), but \(\RR(N_i)\geq K\) and
\(\RR(N_i-M_{i-1})\geq K\). Let $N_i-M_i=S_r,$
where \(z_r\) lies between \(c_{M_i-M_{i-1}}\) and \(c_{M_i}\). Then
\[
c\in f^{S_r}\bigl([c_{M_i-M_{i-1}},c_{M_i}]\bigr)
=
[c_{N_i-M_{i-1}},c_{N_i}].
\]
Because both \(c_{N_i-M_{i-1}}\) and \(c_{N_i}\) lie in \(B(c,\delta/J)\), the two intervals
\(
[c_{N_i-M_{i-1}},c]
\qquad\text{and}\qquad
[c,c_{N_i}]
\)
have length less than \(\delta/J<\delta\). The point \(z_r\) divides
\([c_{M_i-M_{i-1}},c_{M_i}]\) into two subintervals, and \(f^{S_r}\) is monotone on each of them. Applying the contraction principle to each branch gives
\[
|c_{M_i-M_{i-1}}-z_r|<\epsilon/4
\quad \text{ and } \quad
|z_r-c_{M_i}|<\epsilon/4.
\]
Therefore, \(|c_{M_i-M_{i-1}}-c_{M_i}|<\epsilon/2\).
For each \(M_{i-1}<n\leq M_i\), the interval defining \(r_{n-1}(c_1)\) is a monotone image of a subinterval of
\([c_{M_i-M_{i-1}},c_{M_i}]\). Thus \(r_{n-1}(c_1)<\epsilon\).

    Lastly, suppose that $\RR(M_i) \geq K$ and $\RR(M_i - L_i)\geq K$, but $N_i - M_i < K$.  Then $c_{M_i},c_{M_i - s(M_i)} \in B(c,\delta/J)$, and hence $r_{M_i-1}(c_1) < \epsilon$. In the case that $L_i = M_{i-2}$, the contraction principle guarantees that for all $n$ with $M_{i-1} \leq n \leq M_i$, $r_{n-1}(c_1) < \epsilon$. In the case that $N_{i-1}\leq L_i < M_i$, it follows that $r_{n-1}(c_1)< \epsilon$ for all $L_i < n \leq M_i$. Additionally, because $\RR(M_i - L_i)\leq  M_i - L_i$, it follows that the itinerary $I(c_{L_i})$ and the itinerary $I(c_{L_i - M_{i-1}})$ agree for $N_i-L_i$ iterates and both $\RR(c_{L_i})\geq K$ and $\RR(c_{L_i - M_{i-1}}) \geq K$. That is, $c_{L_i},c_{L_i - M_{i-1}}\in B(c,\delta/J)$, and thus $|c_{L_i} - c_{L_i - M_{i-1}}| < 2\delta/J$. By the contraction principle, it will follow that $r_{n-1}(c_1) < \epsilon$ for all $M_{i-1}< n \leq L_i$, and by combining with the above, we see that $r_{n-1}(c_1) < \epsilon$ for all $M_{i-1} < n \leq M_i$.

    We thus conclude that for all $n\geq M_{i-1} + 1$, $r_{n-1}(c_1) < \epsilon$. Since we may do this argument for any $\epsilon > 0$, it follows that $c$ is persistently recurrent.
\end{proof}

\begin{corollary}\label{Corollary to characterization}
    If $c$ is persistently recurrent, then for all $k\in \N$, there exists an $I\in \N$ such that for all $i\geq I$, one of the following holds:
    \begin{enumerate}
        \item $N_i - M_i \geq k$;
        \item $N_i-M_i < k$ but $\RR(M_i)\geq k$ and $M_i-L_i \geq k$; or
        \item $\RR(M_i)<k$ but $\RR(N_i)\geq k$ and $N_i - M_{i-1}\geq k$.
    \end{enumerate}
\end{corollary}
\begin{proof}
    By definition, $M_i - L_i \geq \RR(M_i - L_i)$ and $N_i - M_{i-1}\geq \RR(N_i - M_{i-1})$, so this follows directly from Theorem~\ref{Theorem: Main Symbolic Characterization}.
\end{proof}

\begin{corollary}\label{2nd Corollary to characterization}
    If $c$ is persistently recurrent, then for all $k\in \N$, there exists $I\in \N$ such that for all $i\geq I$, if $\RR(M_i) < k$, then $\RR(M_{i+1})\geq k$.
\end{corollary}
\begin{proof}
    Observe that if conditions \ref{Main 1} and \ref{Main 2} of Theorem~\ref{Theorem: Main Symbolic Characterization} are satisfied, then $\RR(M_i)\geq k$. If $\RR(M_i) < k$, then condition \ref{Main 3} implies that $\RR(N_i)\geq k$; in this case, it follows that $N_i = M_{i+1}$.
\end{proof}

Note that Corollaries~\ref{Corollary to characterization} and \ref{2nd Corollary to characterization} provide useful ways to quickly confirm that $c$ is not persistently recurrent. In general, $M_i-L_i \geq k$ does not imply that $\RR(M_i-L_i)\geq k$, nor does $\RR(M_i)\geq k$ imply either $N_i - M_i \geq k$ or $\RR(M_i - L_i)\geq k$. Hence, we do not claim that either converse holds. We now present symbolic characterizations for (C5) and (C6).

\begin{theorem}\label{Theorem: C5 characterization}
    (C5) holds if and only if for all $M\in \N$, there exists $N\in \N$ such that whenever $n\geq N$ is a cutting or co-cutting time, either $\RR(n-s(n)) > M$ or $\RR(n-\tilde{s}(n)) > M$.
\end{theorem}
\begin{proof}
    First suppose that (C5) holds. Fix $M$ and let $k\in \N$ be chosen such that $S_{k-1} < M\leq S_k$ and set $\epsilon = |z_k-c|$; let $U = B(c,\epsilon)$. Then there exists an $N\in \N$ such that for all $n\geq N$, $U\not\subset \widetilde{D}_n$. Hence, for all cutting and co-cutting times $n\geq N$, it follows that either $c_{n-s(n)}\in U$ or $c_{n-\tilde{s}(n)}\in U$. By Lemma~\ref{Lemma: close to c}, $\RR(n-s(n))> M$ or $\RR(n-\tilde{s}(n))> M$.
    
    Conversely, fix $\epsilon > 0$ and let $M$ be chosen as in Lemma~\ref{Lemma: close to c} so that whenever $\RR(n)>M$, then $|c_n - c| < \epsilon$. Recall that \(\widetilde D_n\) can contain a symmetric neighbourhood of \(c\) only when \(n\) is a cutting or co-cutting time. Let $U = B(c,\epsilon)$ and note if $n$ is not a cutting or co-cutting time, then $U\not\subset \widetilde{D}_n$. Further, by assumption, there exists some $N\in \N$ such that for every cutting or co-cutting time $n\geq N$ it follows that either $\RR(n-s(n))> M$ or $\RR(n-\tilde{s}(n))>M$. Hence, for all cutting and co-cutting times $n\geq N$, either $c_{n-s(n)}\in U$ or $c_{n-\tilde{s}(n)}\in U$. It thus follows that for all $n\geq N$, $\widetilde{D}_n\not\supset U$.
\end{proof}

\begin{corollary}
    Suppose (C5) holds. Then $\max\{N_i - M_i, M_i - L_i\}\to \infty$; that is, $N_i - L_i\to \infty$. 
\end{corollary}

\begin{theorem}\label{Theorem: C6 characterization}
    The map $f$ is not critically monotonic (i.e., (C6) holds)  if and only if for all $M\in \N$, there exists $I\in \N$ such that whenever $i\geq I$ and $\RR(M_i) = S_{t_i} > M$, then either $\RR(M_i - M_{i-1}) = N_i - M_i > M$ or $$\RR(M_i - L_i) = \begin{cases}
        S_{Q(t_{i-2}+1)} > M & \text{ if } M_i = N_{i-1}; \\
        S_{Q(Q(t_{i-1}-1)+1)}>M & \text{ if } M_i > N_{i-1}.
    \end{cases}$$
\end{theorem}
\begin{proof}
    First of all, note that by applying Lemmas~\ref{Lemma: helpful kneading facts} and \ref{Lemma: helpful cutting facts}, we see that $\RR(M_i-L_i)= S_{Q(t_{i-2}+1)}$ when $M_i = N_{i-1}$ and $\RR(M_i - L_i) = S_{Q(Q(t_{i-1}-1)+1)} $ when $M_i > N_{i-1}$.
    
    Suppose first that (C6) holds. Fix $M\in \N$ and let $\epsilon = \epsilon(M)$ be chosen as in Lemma~\ref{Lemma: close to c}. Set $U = B(c,\epsilon)$. Then there exists an $N\in \N$ such that for every $n\geq N$, either $c_n\notin U$ or $\widetilde{D}_n\not\supset U$. Let $I$ be chosen such that $M_I\geq N$. Then for all $i\geq I$, we have by Lemma~\ref{Lemma: close to c} that $c_{M_i}\in U$ if and only if $\RR(M_i)> M$. Additionally, if $c_{M_i}\in U$, then we have that $\widetilde{D}_{M_i}\not\supset U$, so either $c_{M_i - M_{i-1}} \in U$ or $c_{M_i-L_i}\in U$. That is, $\RR(M_i - M_{i-1}) > M$ or $\RR(M_i - L_i) > M$.

    Conversely, suppose that for all $M\in \N$ there exists $I\in \N$ such that whenever $i\geq I$ and $\RR(M_i) > M$, then either $\RR(M_i - M_{i-1}) > M$ or $\RR(M_i - L_i) > M$.  Let $N = M_{I}$. Note that it suffices to consider only the cutting and co-cutting times $n \geq N$, as otherwise it immediately follows that $\widetilde{D}_n\not\supset U$.  Thus, suppose $n\geq N$ is a cutting or co-cutting time. If $n = M_i$ for some $i\in \N$, then it follows by our assumptions that if $\RR(M_i) > M$, then either $\RR(M_i- M_{i-1}) > M$ or $\RR(M_i - L_i)> M$, and thus either $c_{M_i - M_{i-1}}\in U$ or $c_{M_i - L_i}\in U$; in either case, $\widetilde{D}_n\not\supset U$. If $M_i < n < M_{i+1}$, then it follows that $\RR(n) = \RR(n-M_i)$. Thus, $\RR(n) > M$ if and only if $\RR(n-M_i) > M$, which implies that $c_n \in U$ if and only if $c_{n-M_i}\in U$, in which case $\widetilde{D}_n\not\supset U$. Hence (C6) holds.
\end{proof}

\begin{corollary}
    If $N_i - M_i\to \infty$, then (C6) holds.
\end{corollary}

We note that the fact that Theorem~\ref{Theorem: C6 characterization} needs to hold for all $M\in \N$ and not just for large $M$ is important. Consider Example~\ref{Example: Fixed C7 not C6} which satisfied property (C7) but not (C6). Let $M=2$. Then for each co-cutting return time of the form $M_i = 4\cdot 5^n-2$, we have that $\RR(M_i) = 3 > M$, however $\RR(M_i - L_i) = 2$ and $\RR(M_i - M_{i-1}) = 1$. Although the conditions for Theorem~\ref{Theorem: C6 characterization} are satisfied for all $M\geq 3$, the fact it fails for $M=2$ tells us that the map $f$ is critically monotonic.

     \section{A family of unimodal maps with wild Cantor attractors}\label{Section: Absorbing Family}

%{\color{blue} Should we move this section to right after our characterization of persistent recurrence? We could make the narrative about the relation persistent recurrence has with the existence of absorbing cantor sets to be more cohesive. We could then end the paper with a section on interesting examples of unimodal maps with various types of recurrence and other interesting behaviors (what is currently sections 5 and 6). Where does Section 4 fit in this?} \Jernej{I have added a few papers that are connected to this section in Overleaf. Section 4 in my opinion also fits in the study of different types of recurrence and since it implies CE condition it is interesting also from measure theoretic point of view (see the uploaded paper of Avila). I agree that the narrative of the paper should be around persistent recurrence and Wild cantor attractors; we could phrase it in a way that we offer a new viewpoint on the problem of classifying Wild Cantor attractors using symbolic dynamics and besides that also clear up some relationships and confusions between different types of recurrence which are important to understand the overall picture within unimodal maps. I am not sure if Wild Cantor attractors beyond Fibonacci maps have actually been explicitly described in the literature, do you know? I plan to dig in the literature in more detail soon.}

In this section we provide a technique for generating kneading sequences of unimodal maps with absorbing Cantor sets. Each kneading sequence will correspond to a non-renormalizable unimodal map with an embedded odometer; we will then apply \cite[Theorem 6.1]{Bruin_Absorbing_Cantor}. This section is motivated by the example provided in \cite[Section 6]{LiShen}, which demonstrated the existence of a unimodal map with embedded odometer that has an absorbing Cantor set. The technique presented here is a modification of the construction of kneading sequences belonging to maps with embedded odometers in \cite{Alvin_strange_star}.

\begin{definition}
    Let $\gamma = (j_1,j_2, \ldots)$ to be a sequence of integers with each $j_i\geq 3$ and let $H_0C$ and $E_0C$ be two finite shift-maximal sequences that differ in precisely one position. We refer to $(H_0,E_0,\gamma)$ as an \emph{kneading odometer triple}. We will let $\Delta=1$ if $H_0$ has even parity and $\Delta = 0$ if $H_0$ has odd parity; set $\Delta' = 1-\Delta$.
\end{definition}

\begin{lemma}\label{Lemma: star product is shift-maximal}
    Let $H_0$ and $E_0$ be as in the definition of a kneading odometer triple and let $j\geq 3$. Then $$H_1 C = H_0 (\Delta E_0)^{j-2} \Delta H_0 C \quad \text{ and } \quad E_1C = H_0(\Delta E_0)^{j-2} \Delta' H_0 C$$ are shift-maximal. Further, $H_1$ has the opposite parity as $H_0$ and $E_1$ has the opposite parity as $E_0$. 
\end{lemma}
\begin{proof}
    % \sout{The proof follows from small modifications to the proofs of Propositions 3.1 and 3.2 in \cite{Alvin_strange_star}, which showed that $H_0 (\Delta E_0)^{j-2}\Delta E_0 C$ and $H_0 (\Delta  E_0)^{j-2}\Delta' E_0 C$ are both shift-maximal whenever $j\geq 2$. Because $E_0$ and $H_0$ differ in only one position, we would have similar cases to consider as in \cite{Alvin_strange_star}, but with only one symbol changed.}

    The proof is identical to the proofs of Propositions 3.1 and 3.2 in \cite{Alvin_strange_star}, with \(E_0\) replaced by \(H_0\) in the terminal block. Since \(H_0\) and \(E_0\) differ in exactly one symbol, the parity-lexicographic comparisons are unchanged except at that position, where the choice of \(\Delta\) and \(\Delta'\) gives the required inequality.
\end{proof}

\begin{definition}
Let \(\Delta_n=1\) if \(H_{n-1}\) has even parity and \(\Delta_n=0\) otherwise, and set \(\Delta_n'=1-\Delta_n\).
Given the kneading odometer triple $(H_0,E_0,\gamma)$, we may define an infinite kneading sequence 
$$H_\infty = \lim_{n\to \infty} H_n C$$
where $H_n = H_{n-1}(\Delta_n E_{n-1})^{j_n-2}\Delta_n H_{n-1}$ and $E_n = H_{n-1}(\Delta_n E_{n-1})^{j_n-2}\Delta_n'H_{n-1}$ for all $n\geq 1$.
\end{definition}

\begin{comment}
\begin{theorem}\label{Theorem: infinite strange star}
  The infinite kneading sequence $H_\infty$ is the kneading sequence of a non-infinitely renormalizable map $f$ where $f|_{\omega(c)}$ is conjugate to the odometer given by $\gamma' = (|H_0|+1, j_1,j_2,\ldots)$; moreover, if $H_0C$ belongs to a non-renormalizable map, then $f$ is non-renormalizable.
 \end{theorem}
 \begin{proof}
 By repeated application of Lemma~\ref{Lemma: star product is shift-maximal}, it follows that $H_\infty$ is a shift-maximal sequence. Because $H_\infty$ is constructed to be Toeplitz and have the finite time containment property, it follows that $f|_{\omega(c)}$ is conjugate to an odometer \cite{Alvin_ftcp}. Because of the Toeplitz structure of $H_\infty$, it is easy to check that the underlying odometer is generated by $\gamma'$. Lastly, by applying a similar argument as in \cite[Theorem 3.9]{Alvin_strange_star}, we can show that $H_\infty$ does not belong to an infinitely renormalizable map; further $f$ is non-renormalizable whenever $H_0C$ belongs to a non-renormalizable map.
\end{proof}

\Jernej{What if we replace the above theorem and its proof with the following}
\end{comment}

\begin{theorem}\label{Theorem: infinite strange star} The sequence \(H_\infty\) is shift maximal. Hence there exists a unimodal map
\(f\) such that
$\mathcal K(f)=H_\infty$.
Moreover, \(f|_{\omega(c)}\) is conjugate to the odometer generated by
\[
\gamma'=(|H_0|+1,j_1,j_2,\ldots).
\]
The map \(f\) is not infinitely renormalizable. Furthermore, if \(H_0C\) is the
kneading sequence of a non-renormalizable unimodal map, then \(f\) is
non-renormalizable.
\end{theorem}

\begin{proof}
By construction, \(H_nC\) and \(E_nC\) differ in precisely one position for every
\(n\). Hence Lemma~\ref{Lemma: star product is shift-maximal} applies
inductively, and each finite word \(H_nC\) is shift-maximal. Since shift-maximality is determined by finite parity-lexicographic comparisons and the
words \(H_nC\) converge prefixwise to \(H_\infty\), the limit \(H_\infty\) is
shift-maximal. Therefore \(H_\infty\) is realized as the kneading sequence of a
unimodal map.

The recursive construction gives a Toeplitz sequence. Moreover, the same
argument as in \cite{Alvin_strange_star} shows that \(H_\infty\) has the finite
time containment property. It follows from \cite{Alvin_ftcp} that
\(f|_{\omega(c)}\) is conjugate to an odometer. The periods of the Toeplitz
skeleton are
\[
|H_0|+1,\quad (|H_0|+1)j_1,\quad (|H_0|+1)j_1j_2,\quad \ldots,
\]
so the associated odometer is generated by
\[
\gamma'=(|H_0|+1,j_1,j_2,\ldots).
\]
Finally, by the argument of \cite[Theorem~3.9]{Alvin_strange_star},
\(H_\infty\) does not correspond to an infinitely renormalizable map. The same
argument from \cite[Theorem~3.9]{Alvin_strange_star} shows that, if \(H_0C\) is the kneading sequence of a non-renormalizable
unimodal map, then the map with kneading sequence \(H_\infty\) is
non-renormalizable.
\end{proof}

\begin{comment}
\begin{remark}\label{Remark: cutting times H}
    Let $(H_0,E_0,\gamma)$ be a kneading odometer triple. Let $|H_{n}C| = T_{n} = (|H_0|+1)\cdot j_1\cdots j_{n}$. Then $T_{n}$ is a cutting time for $H_\infty$ for each $n\in \N$. Additionally, the cutting times in $H_\infty$  that are between $T_{n-1}$ and $T_{n}$ are  $T_{n-1}+ (j_{n-2}-1)T_{n-2}$ , $2T_{n-1}$, $2T_{n-1}+ (j_{n-2}-1)T_{n-2}$, $3T_{n-1}$, $3T_{n-1}+ (j_{n-2}-1)T_{n-2}$, $\ldots$ , $(j_{n-1}-1)T_{n-1}$, and $T_n$.
\end{remark}

\Jernej{What if we extend the previous remark to lemma below?}
\end{comment}

\begin{lemma}\label{Remark: cutting times H}
Let \((H_0,E_0,\gamma)\) be a kneading odometer triple, and set
\[
T_n:=|H_nC|=(|H_0|+1)j_1j_2\cdots j_n .
\]
Then \(T_n\) is a cutting time for \(H_\infty\) for every \(n\ge0\). Writing
\(T_n=S_{t_n}\), we have
\[
t_n=t_{n-1}+2j_n-3
\qquad\text{for all } n\ge1.
\]
Moreover, for \(n\ge2\), the cutting times between \(T_{n-1}\) and \(T_n\), together with \(T_n\), are
\[
qT_{n-1}+(T_{n-1}-T_{n-2})
\quad\text{and}\quad
(q+1)T_{n-1},
\qquad 1\le q\le j_n-2,
\]
together with \(T_n=j_nT_{n-1}\).
\end{lemma}

\begin{proof}

By construction,
\(
H_nC
=
H_{n-1}(\Delta_nE_{n-1})^{j_n-2}\Delta_nH_{n-1}C .
\)
Since \(|H_{n-1}C|=|E_{n-1}C|=T_{n-1}\), we have
\(
|H_{n-1}|=|E_{n-1}|=T_{n-1}-1.
\)
Thus
\[
|H_nC|
=
(T_{n-1}-1)+(j_n-2)T_{n-1}+(T_{n-1}+1)
=
j_nT_{n-1}.
\]
Hence
\(
T_n=(|H_0|+1)j_1\cdots j_n.
\)
Since \(H_nC\) is an initial block of \(H_\infty\) ending in \(C\), the time
\(T_n\) is a cutting time.

The decomposition of \(H_nC\) shows that, between \(T_{n-1}\) and \(T_n\), new
cutting times occur immediately after each repeated \(\Delta_nE_{n-1}\)-block and
at the corresponding internal positions determined from the previous stage. These
positions are precisely
\[
qT_{n-1}+(T_{n-1}-T_{n-2})
\quad\text{and}\quad
(q+1)T_{n-1},
\qquad 1\le q\le j_n-2,
\]
followed by \(T_n=j_nT_{n-1}\). Thus the number of new cutting times from
\(T_{n-1}\) to \(T_n\) is
\(
2(j_n-2)+1=2j_n-3.
\)
Therefore, if \(T_n=S_{t_n}\), then
\(
t_n=t_{n-1}+2j_n-3.
\)
\end{proof}

\begin{theorem}\label{Theorem: strange star product satisfies bruin conditions}
    Let $(H_0,E_0,\gamma)$ be a kneading odometer triple and let $H_\infty$ be the generated kneading sequence. If there exists some $M\in \N$ such that each $j_i$ from $\gamma$ is such that $j_i\leq M$, then there exist $k_0,N\in \N$ such that for all $k\geq k_0$ 
    \begin{itemize}
        \item $Q(k) \geq k-N$ and
        \item $Q(k+1) > Q^2(k) + 1$.
    \end{itemize}
\end{theorem}
\begin{comment}
\begin{proof}
 For ease, set $T_n = |H_nC| = S_{t_n}$ for all $n\in \N$. By Remark~\ref{Remark: cutting times H}, $t_{n+1} = t_n + 2j_n-3$ for all $n\in \N$. We have the following cutting times for all $n\geq 2$: $$Q(k) = \begin{cases}
     t_{n-1} & \text{ if } k = t_n, \\
     t_n - 1 & \text{ if } k = t_n + 2l-1 \text{ for } 1\leq l \leq j_n - 2, \\
     t_{n-1} & \text{ if } k = t_n + 2l \text{ for } 1\leq l \leq j_n - 2.
 \end{cases}$$

 We additionally have $Q^2(k) = t_{n-2}$ for all $k = t_n +j$ with $0\leq j \leq j_n - 2$.  Thus, it holds for all $k$ large enough, that $Q(k+1) > Q^2(k) + 1$.

 Lastly, consider the following values for $k - Q(k)$:
 $$k - Q(k) = \begin{cases}
     t_n-t_{n-1} = 2j_{n-1}-3 & \text{ if } k = t_n, \\
     t_n + (2l-1) - (t_n-1) = 2l & \text{ if } k = t_n + 2l-1 \text{ for } 1\leq l \leq j_n - 2, \\
     t_n + 2l - t_{n-1} = 2l + 2j_{n-1}-3 & \text{ if } k = t_n + 2l \text{ for } 1\leq l \leq j_n - 2.
 \end{cases}$$ Under the assumption that $j_n \leq M$ for all $n$, it follows that for all $k$ large enough, $k - Q(k) \leq 4M - 7$. Hence, let $N = 4M-7$.
\end{proof}

\Jernej{Given we extend the previous remark to lemma, we might need to extend the proof as below?}

\end{comment}

\begin{proof}

Set
\(
T_n=|H_nC|=S_{t_n}.
\)
By Lemma~\ref{Remark: cutting times H}, we have
\(
t_n=t_{n-1}+2j_n-3.
\)
Moreover, for \(n\geq2\), the cutting times between \(T_{n-1}\) and \(T_n\) are
\[
qT_{n-1}+(T_{n-1}-T_{n-2})
\quad\text{and}\quad
(q+1)T_{n-1},
\qquad 1\leq q\leq j_n-2,
\]
followed by \(T_n=j_nT_{n-1}\).

Using the relation
\(
S_k-S_{k-1}=S_{Q(k)},
\)
we obtain, for \(n\geq2\),
\[
Q(k)=
\begin{cases}
t_{n-1}-1, & \text{if } k=t_{n-1}+2q-1,\quad 1\leq q\leq j_n-2,\\
t_{n-2}, & \text{if } k=t_{n-1}+2q,\quad 1\leq q\leq j_n-2,\\
t_{n-1}, & \text{if } k=t_n.
\end{cases}
\]

Indeed, the corresponding differences between consecutive cutting times are
\(
T_{n-1}-T_{n-2}=S_{t_{n-1}-1},
\)
\(
T_{n-2}=S_{t_{n-2}},
\)
and
\(
T_{n-1}=S_{t_{n-1}},
\)
respectively.
We first prove that \(Q(k)\geq k-N\) for all sufficiently large \(k\).
Assume \(j_n\leq M\) for all \(n\). In the three cases above, we have
\[
k-Q(k)=
\begin{cases}
2q, & \text{if } k=t_{n-1}+2q-1,\\
(t_{n-1}-t_{n-2})+2q, & \text{if } k=t_{n-1}+2q,\\
t_n-t_{n-1}, & \text{if } k=t_n.
\end{cases}
\]
Since
\(
t_n-t_{n-1}=2j_n-3,
\)
it follows that
\(
k-Q(k)\leq 4M-7
\)
for all sufficiently large \(k\). Thus we may take
\(
N=4M-7.
\)

It remains to check that
\(
Q(k+1)>Q^2(k)+1
\)
for all sufficiently large \(k\). Let \(k\) lie in the block between
\(t_{n-1}\) and \(t_n\), with \(n\) sufficiently large.
If
\(
k=t_{n-1}+2q-1,
\)
then \(Q(k)=t_{n-1}-1\), while
\(
Q(k+1)=t_{n-2}.
\)
Moreover, \(Q^2(k)=Q(t_{n-1}-1)=t_{n-3}\). Hence
\(
Q(k+1)=t_{n-2}>t_{n-3}+1=Q^2(k)+1
\)
for all sufficiently large \(n\).
If
\(
k=t_{n-1}+2q,
\)
then \(Q(k)=t_{n-2}\), so \(Q^2(k)=t_{n-3}\). If \(k+1<t_n\), then
\(
Q(k+1)=t_{n-1}-1,
\)
and if \(k+1=t_n\), then
\(
Q(k+1)=t_{n-1}.
\)
In either case,
\(
Q(k+1)>t_{n-3}+1=Q^2(k)+1
\)
for all sufficiently large \(n\).
Finally, if \(k=t_n\), then
\(
Q(k)=t_{n-1}
\) and 
\(Q^2(k)=t_{n-2},
\)
while \(k+1=t_n+1\) is the first new index in the next block, so
\(
Q(k+1)=t_n-1.
\)
Thus
\(
Q(k+1)=t_n-1>t_{n-2}+1=Q^2(k)+1
\)
for all sufficiently large \(n\).

Therefore there exist \(k_0,N\in\mathbb N\) such that for all \(k\geq k_0\),
\(
Q(k)\geq k-N
\) and
\(
Q(k+1)>Q^2(k)+1.
\)
\end{proof}

\begin{corollary}\label{cor:Wild attractors dense}
    With the constraints as listed in Theorem~\ref{Theorem: strange star product satisfies bruin conditions}, whenever $f$ is a unimodal map with kneading sequence $H_\infty$, we can find an order $l_0 = l_0(N)$ such that if  the S-unimodal map $f$ has order $l\geq l_0$, then $f$ has an absorbing Cantor set.
\end{corollary}
\begin{proof}
    Combine together Theorem~\ref{Theorem: strange star product satisfies bruin conditions} and \cite[Proposition 6.1]{Bruin_Absorbing_Cantor}.
\end{proof}

\begin{remark}
    In Lemma~\ref{Lemma: star product is shift-maximal} and Theorem~\ref{Theorem: infinite strange star}, if we were to permit $j=j_n=2$ for all $n\in \N$, then $H_\infty = H_0\star \mathcal{K}(g)$ where $g$ is the Feigenbaum, or $2^\infty$, map (i.e., the unique map in the logistic family with periodic points of period $2^n$ for all $n$ and no other periodic points) and $\star$ represents the star product (see \cite{ColletEckmann}). In this case, the map $f$ with $\mathcal{K}(f) = H_\infty$ would be infinitely renormalizable. The turning point would be persistently recurrent; however $f$ would not have an absorbing Cantor set.
\end{remark}

\begin{corollary}\label{cor:Wild attractors odometer}
Let \(\gamma=(j_1,j_2,\ldots)\) be a sequence of integers with \(j_i\ge2\).
Assume that only finitely many primes divide the entries \(j_i\).
Then there exist uncountably many pairwise non-conjugate \(S\)-unimodal maps \(f\)
such that \(f|_{\omega(c)}\) is conjugate to the odometer generated by \(\gamma\)
and \(f\) has an absorbing Cantor set.
\end{corollary}

\begin{proof}
Two sequences generate conjugate odometers if one can be obtained from the other
by refining or grouping entries. Hence we may replace \(\gamma\) by an equivalent
sequence obtained by factoring each \(j_i\) into primes and regrouping these
factors into blocks. Since only finitely many primes occur, we can construct a
sequence
\[
\gamma'=(k_1,k_2,\ldots)
\]
such that \(\gamma'\) generates the same odometer as \(\gamma\) and there exists
\(M\in\mathbb N\) with
\[
3\le k_i\le M \qquad \text{for all } i.
\]

Let \(\mathcal K(f)=e_1e_2e_3\cdots\) be the kneading sequence of a
non-renormalizable unimodal map \(f\). By the construction in
\cite{Alvin_strange_star} (see also \cite{BlockKeeslingMisiurewicz}), for each
\(m\in\mathbb N\) we can choose finite words \(H_0\) and \(E_0\) such that:
\begin{itemize}
    \item \(H_0C\) and \(E_0C\) are shift maximal,
    \item they differ in exactly one position,
    \item they agree with \(\mathcal K(f)\) in the first \(m\) symbols,
    \item and \(|H_0C|=k_1k_2\cdots k_r\) for some \(r\in\mathbb N\).
\end{itemize}

Applying Theorem~\ref{Theorem: infinite strange star} to the kneading odometer
triple
$(H_0,E_0,(k_{r+1},k_{r+2},\ldots))$
produces a unimodal map \(\tilde f\) such that
\(\tilde f|_{\omega(c)}\) is conjugate to the odometer generated by
\[
(|H_0|+1,k_{r+1},k_{r+2},\ldots)
=
(k_1,k_2,\ldots),
\]
which is conjugate to the odometer generated by \(\gamma\).
Moreover, since the entries \(k_i\) are uniformly bounded, it follows from
Corollary~\ref{cor:Wild attractors dense} that, for sufficiently large critical
order, the map \(\tilde f\) has an absorbing Cantor set.

Finally, let \(\{\mathcal K(f_\alpha)\}_{\alpha\in A}\) be an uncountable family
of pairwise distinct kneading sequences of non-renormalizable unimodal maps.
For each \(\alpha\), perform the above construction using \(\mathcal K(f_\alpha)\)
to obtain a map \(\tilde f_\alpha\). The kneading sequence of \(\tilde f_\alpha\)
agrees with \(\mathcal K(f_\alpha)\) on arbitrarily long initial blocks, so
distinct sequences \(\mathcal K(f_\alpha)\) yield distinct kneading sequences.
Since kneading sequences determine unimodal maps up to topological conjugacy,
the maps \(\tilde f_\alpha\) are pairwise non-conjugate.

Therefore there exist uncountably many pairwise non-conjugate \(S\)-unimodal maps
with the required properties.
\end{proof}

%\Jernej{Do we need to argue why this $Q(k)$ actually corresponds to the kneading sequence $H_{\infty}$?}\red{I think since the cutting times are coming from the new lemmas that we added more detail to, we should be fine to just state the kneading map. The first few cutting times are found by inspection, and then the ones that repeat are from Lemma 4.5 with the fact that this is a nice fixed quadrupling procedure.}\Jernej{I agree.}

We now include an example of a unimodal map with an absorbing Cantor set that can be constructed by the technique introduced in this section. We state the kneading odometer triple, the resulting kneading map, and a corresponding critical order that guarantees the existence of the absorbing Cantor set.

\begin{example}\label{Example: absorbing Cantor set}
    Let $H_0 = 100011$, $E_0 = 100111$, and $\gamma = (4,4,4,4,\ldots)$. Then the kneading sequence $H_\infty$ generated as in Theorem~\ref{Theorem: infinite strange star} has kneading map given  by $$Q(k) = \begin{cases}
        0, & k\in \{1,2,3,5,8,11\} \\
        1, & k\in \{4,7,10\} \\
        3, & k\in \{6,9\} \\
        5, & k\in \{12,14,16\} \\
        11, & k\in \{13,15\} \\
        k-5, & k = 5n-3, n\geq 4\\
        k-2, & k = 5n-2, n\geq 4\\
        k-7, & k=5n-1, n\geq 4\\
        k-4, & k= 5n, n\geq 4\\
        k-9, & k =5n+1, n\geq 4.
    \end{cases}$$
    Additionally, if $\mathcal{K}(f) = H_\infty$, then $f|_{\omega(c)}$ is topologically conjugate to the odometer generated by $\gamma'=(7,4,4,4,\ldots)$. We note that by construction, $f$ is non-renormalizable.

    Observe that for all $k\geq 17$, $$Q(k+1) = \begin{cases}
        k-1, & k = 5n-3, n\geq 4 \\
        k-6, & k = 5n-2, n\geq 4 \\
        k-3, & k = 5n-1, n\geq 4 \\
        k-8, & k = 5n, n\geq 4 \\
        k-4, & k = 5n+1, n\geq 4
    \end{cases}$$

    and for $k\geq 22$, 
    $$Q^2(k) + 1 = \begin{cases}
        k-9, & k = 5n-3, n\geq 5 \\
        k-10 & k = 5n-2, n\geq 5 \\
        k-11 & k = 5n-1, n\geq 5 \\
        k-12 & k = 5n, n\geq 5 \\
        k-13 & k = 5n+1, n\geq 5.
    \end{cases}$$

    Hence, let $k_0 = 22$. Then for all $k\geq k_0$, it follows that $Q(k+1) > Q^2(k)+1$ and $Q(k) \geq k-9$.  It thus holds that there exists an $l_0 = l_0(9)$ such that if $l\geq l_0$ is the critical order of the S-unimodal map $f$ with $\mathcal{K}(f) = H_\infty$, then $f$ has an absorbing Cantor set.
    
\end{example}

We believe this process can be modified to other ``strange star products" that can be used to generate additional kneading sequences of unimodal maps with embedded odometers and absorbing Cantor sets. That is, we could modify the process to allow for different rules for concatenating the $H_n$'s and $E_n$'s or even allow for the concatenation of more than two distinct blocks at each stage. It remains open whether this construction can be modified to allow for the sequence of $j_i$'s in $\gamma$ to have arbitrarily unbounded prime divisors. Our proof that this construction results in unimodal maps with absorbing Cantor sets relies on applying \cite[Proposition 6.1]{Bruin_Absorbing_Cantor}, and thus it is necessary for us to keep a bound on the number of cutting times between $|H_nC|$ and $|H_{n+1}C|$; it would be interesting if there is a construction that allows for the $j_i$'s to be unbounded yet still preserves this property. 

It is well-known that if a unimodal map has an absorbing Cantor set, then the turning point $c$ is persistently recurrent \cite{Blokh_Lyubich_Measurable_Dynamics_of_S_unimodal}. Additionally, Guckenheimer and Johnson show that if $f$ has an absorbing Cantor set, then $c$ is not critically monotonic \cite{GJ}; that is, (C6) must hold. There is no complete characterization for the existence of an absorbing Cantor set. It is tempting to ask whether every kneading sequence belonging to a unimodal map with an embedded strange odometer for which the turning point is persistently recurrent has combinatorics that allows for the existence of an absorbing Cantor set. This is not the case, as Example~\ref{Example: Fixed C7 not C6} provides the kneading sequence of a unimodal map with an embedded strange odometer that is  persistently recurrent and critically monotonic; thus, the kneading sequence does not belong to any unimodal map with an absorbing Cantor set. Therefore, it is natural to ask Question~\ref{Question: wild attractors}.

\begin{comment}
We thus conclude this section with the following question.

\begin{question}
    Does every kneading sequence belonging to a unimodal map with an embedded strange odometer for which the turning point is not critically monotonic (i.e., (C6) holds) have combinatorics that allows for the existence of an absorbing Cantor set as long as the associated map has high enough order? If not, what if we assume either (C5) or (C4) holds?
\end{question}
\end{comment}

% \begin{remark} Things to comment on:
% \begin{itemize}
%     \item See if Henk's result generalizes when $f$ is finitely renormalizable. Does non-renormalizable assumption actually matter?
% \end{itemize}
% \end{remark}

%
%
%
%
%
%
%
%
 \section{Slow Recurrence and Collet-Eckmann condition}\label{Section:slow recurrence}
Let $\{M_i\}$ be the sequence of return times for the kneading sequence $\mathcal{K}(f)$; see Section~\ref{Section: Symbolic Characterization} for additional background on the function $\RR(n)$.

\begin{definition}\cite[Lemma 2.2]{Sands}\label{def:slow recurrence} We say that the kneading sequence $\mathcal{K}(f)$ is {\em slowly recurrent} if either of the two equivalent statements holds.

\begin{equation}
    \lim_{l\to \infty}\limsup_{i\to \infty} \frac{\sum_{j=1}^i \begin{cases}
        \mathcal{R}(j) & \text{if } \mathcal{R}(j) \geq l \\ 
        0 & \text{otherwise}
    \end{cases}}{i} = 0
\end{equation}

\begin{equation}
    \lim_{l\to \infty}\limsup_{i\to \infty} \frac{\sum_{j=1}^i \begin{cases}
        \mathcal{R}(M_j) & \text{if } \mathcal{R}(M_j) \geq l \\ 
        0 & \text{otherwise}
    \end{cases}}{M_i} = 0
\end{equation}
\end{definition}

The following proposition follows from \cite[Theorem 2]{Bruin_QuasiSymmetry}; we include a proof using symbolic arguments rather than the measure-theoretic arguments provided in \cite{Bruin_QuasiSymmetry}. 

\begin{proposition}\label{Proposition: Persitently recurrent not slowly recurrent}
    If $c$ is persistently recurrent, then $\mathcal{K}(f)$ is not slowly recurrent.
\end{proposition}
\begin{proof}
     Suppose $c$ is persistently recurrent. Then by Corollary~\ref{2nd Corollary to characterization}, it follows that for every $k\in \N$, there exists an $I_k\in \N$ such that for all $i\geq I_k$, either $\RR(M_i)\geq k$ or $\RR(M_{i+1})\geq k$. Without loss of generality, let $I_k$ be chosen such that $\RR(M_{I_k})\geq k$. We claim that for infinitely many $i$, $$\sum_{j=I_k}^i\begin{cases}
        \mathcal{R}(M_j) & \text{if } \mathcal{R}(M_j) \geq k \\ 
        0 & \text{otherwise}
    \end{cases} \quad \geq N_{i+1}-M_{I_k}.$$

    First observe that $\RR(M_{I_k}) = N_{I_k+1}- M_{I_k}$. Suppose that $i$ is chosen such that $$\sum_{j=I_k}^i\begin{cases}
        \mathcal{R}(M_j) & \text{if } \mathcal{R}(M_j) \geq k \\ 
        0 & \text{otherwise}
    \end{cases} \quad \geq N_{i+1}-M_{I_k}.$$ Then if $\RR(M_{i+1})\geq k$, it follows that adding $\RR(M_{i+1})$ to the previous sum will result in a total that is at least as large as $N_{i+1}-M_{I_k} + N_{i+2}-M_{i+1}$. Because $N_{i+1}\geq M_{i+1}$, we conclude that the sum is larger than $N_{i+2} - M_{I_k}$. 

    If $\RR(M_{i+1})< k$, then it follows that $M_{i+2} = N_{i+1}$ and $\RR(M_{i+2})\geq k$. Thus $$\sum_{j=I_k}^{i+2}\begin{cases}
        \mathcal{R}(M_j) & \text{if } \mathcal{R}(M_j) \geq k \\ 
        0 & \text{otherwise}
    \end{cases} \quad \geq N_{i+1}-M_{I_k} + N_{i+3} - N_{i+1} = N_{i+3}-M_{I_k}.$$
    Hence $$\limsup_{i\to \infty}\frac{\sum_{j=1}^i\begin{cases}
        \mathcal{R}(M_j) & \text{if } \mathcal{R}(M_j) \geq k \\ 
        0 & \text{otherwise}
    \end{cases}}{M_i} \geq \limsup_{i\to\infty} \frac{N_{i+1}}{M_i} - \frac{M_{I_k}}{M_i} \geq 1.$$ We thus conclude that $$\lim_{l\to \infty}\limsup_{i\to \infty} \frac{\sum_{j=1}^i \begin{cases}
        \mathcal{R}(M_j) & \text{if } \mathcal{R}(M_j) \geq l \\ 
        0 & \text{otherwise}
    \end{cases}}{M_i} > 0.$$
    
    Hence $\mathcal{K}(f)$ is not slowly recurrent.
\end{proof}

% \begin{theorem}
%     If (C5) holds, then $c$ is not slowly recurrent.
% \end{theorem}
% \begin{proof}
%     Suppose that (C5) holds. This implies that $N_{i+1}-M_i\to \infty$, or equivalently $\RR(M_i) \to \infty$. Fix $l\in \N$. Then there exists $N_l\in \N$ such that for all $i\geq N$, $\RR(M_i)\geq l$.  Fix $i\in \N$.  Then 
%     \begin{align*}
%         \frac{\sum_{j=1}^i\begin{cases}
%         \mathcal{R}(M_j) & \text{if } \mathcal{R}(M_i) \geq l \\ 
%         0 & \text{otherwise}
%     \end{cases}}{M_i} &\geq \frac{\sum_{j=N_l}^i \RR(M_j)}{M_j} \\
%     & = \frac{\RR(M_{N_l}) + \RR(M_{N_l+1}) + \cdots + \RR(M_i)}{M_i} \\
%     & > \frac{(M_{N_l+1}-M_{N_l})+(M_{N_l+2} - M_{N_l+1}) +-\cdots + (M_{i+1}- M_i)}{M_i}\\
%     &= \frac{M_{i+1}-M_{N_l}}{M_i}
%     \end{align*}

%     Hence $$\limsup_{i\to \infty} \frac{\sum_{j=1}^i \begin{cases}
%         \RR(M_j) & \text{if } \RR(M_j) \geq l \\
%         0 & \text{otherwise}
%     \end{cases}}{M_i} \geq \limsup_{i\to \infty} \frac{M_{i+1}}{M_i} - \frac{M_{N_l}}{M_i} \geq 1.$$  We thus conclude that $$\lim_{l\to \infty}\limsup_{i\to \infty} \frac{\sum_{j=1}^i \begin{cases}
%         \mathcal{R}(M_j) & \text{if } \mathcal{R}(M_j) \geq l \\ 
%         0 & \text{otherwise}
%     \end{cases}}{M_i} > 0.$$ Hence $\mathcal{K}(f)$ is not slowly recurrent.
%     \end{proof}

\begin{lemma}\label{Lemma:Long co-cuts not slowly recurrent}
    If for infinitely many co-cutting times $n_i$, it follows that $\mathcal{R}(n_i)\geq n_i$, then $\mathcal{K}(f)$ is not slowly recurrent.
\end{lemma}
\begin{proof}
    Suppose there is an infinite sequence of co-cutting times $(m_n)$ such that $\mathcal{R}(m_n) \geq m_n$.  Then for each fixed $l\in \N$, there exists an $N_l\in \N$ such that for all $n\geq N_l$, $\frac{\mathcal{R}(m_n)}{m_n} \geq 1$.  Hence $\mathcal R(m_n)\geq m_n\geq l$, and therefore $$\lim_{m_n\to\infty} \frac{\sum_{j=1}^{m_n} \begin{cases} \mathcal{R}(j) & \text{ if }\mathcal{R}(j) \geq l \\ 0 \text{ else } \end{cases}}{m_n} \geq \lim_{m_n\to \infty} \frac{\mathcal{R}(m_n)}{m_n} \geq 1.$$ We conclude that $$\lim_{l\to \infty}\limsup_{i\to \infty} \frac{\sum_{j=1}^i \begin{cases}
        \mathcal{R}(j) & \text{if } \mathcal{R}(j) \geq l \\ 
        0 & \text{otherwise}
    \end{cases}}{i} \neq 0.$$ It thus follows that $\mathcal{K}(f)$ is not slowly recurrent.
\end{proof}

\begin{definition}\label{Definition: Linear recurrence}
Recall that an infinite sequence $\overline{x}$ is linearly recurrent if there exists some $L\in \N$ such that every finite word $w$ appearing in $\overline{x}$ appears with gap $\leq L\cdot|w|$. 
\end{definition}

We now show that linear recurrence is incompatible with slow recurrence.

\begin{proposition}\label{prop: linearly recurrent not slowly recurrent}
    If $f$ is a unimodal map such that $\mathcal{K}(f)$ is a linearly recurrent sequence, then $\mathcal{K}(f)$ is not slowly recurrent.
\end{proposition}
\begin{proof}
    Suppose that $\mathcal{K}(f)$ is linearly recurrent. Then there exists $L\in \N$ such that every finite word $w$ appears in $\mathcal{K}(f)$ with gap $\leq L\cdot|w|$. Fix $l\in \N$. 
 Every occurrence of the word $e_1e_2\cdots e_l$ beginning at position $m+1$
gives agreement of length at least $l$ between $\mathcal K(f)$ and its shift by
$m$. Hence $\mathcal R(m)\geq l$. Since such occurrences have gaps at most
$Ll$, the set of such integers $m$ has lower density at least $1/(Ll)$.
Therefore,
\[
\limsup_{i\to\infty}
\frac{
\sum_{j=1}^i
\begin{cases}
\mathcal R(j) & \text{if } \mathcal R(j)\geq l,\\
0 & \text{otherwise}
\end{cases}
}{i}
\geq
l\cdot \frac{1}{Ll}
=
\frac{1}{L}.
\]
    As the above limit yields $\frac{1}{L}$ for all lengths $l\in \N$, it follows that 
    \[
\lim_{l\to\infty}\limsup_{i\to\infty}
\frac{
\sum_{j=1}^i
\begin{cases}
\mathcal R(j) & \text{if } \mathcal R(j)\geq l,\\
0 & \text{otherwise}
\end{cases}
}{i}
\geq \frac1L>0.
\]  and thus $\mathcal{K}(f)$ is not slowly recurrent.   
\end{proof}

\begin{definition}
Let $f$ be an $S$-unimodal map with turning point $c$. Map $f$ satisfies the \emph{Collet-Eckmann condition (CE-condition)} provided there exists $\kappa>0$ and $\lambda>1$ so that
$$|(f^n)'(f(c))|>\kappa \lambda^n$$ for all $n\geq 1$.
\end{definition}

Nowicki and Przytycki \cite{NowickiPrzytycki1998} proved that within $S$-unimodal maps the CE-condition is equivalent to the following topological condition called finite criticality (also sometimes referred to as topological CE-condition).

\begin{definition}
Let $f:I\to I$ be unimodal with critical point $c$. We call $f$ {\em critically finite} provided there exist $M>0$, $P>0$ and $\delta>0$ such that for every $x\in I$ there exist a strictly increasing sequence of positive integers $\{n_i\}_{i\geq 1}$ such that for each $i$ we have $n_i\leq P\cdot i$ and
$$ \#\{j| 0\leq j\leq n_i \text{ and } c\in Comp(f^j(x),f^{-(n_i-j)}(B(f^{n_i}(x),\delta)))\}\leq M.$$
\end{definition}

For $S$-unimodal maps, note that by \cite[Corollary 36]{Sands} slow recurrence (a combinatorial condition) implies the CE-condition (a metric condition). In \cite[Theorem 2]{Bruin_QuasiSymmetry}, it was shown that an  S-unimodal map $f$ with a persistently recurrent turning point does not satisfy the CE-condition. 

In personal correspondence with Henk Bruin, we were asked if it was possible to find an S-unimodal map $f$ that satisfies the CE-condition with $f|_{\omega(c)}$ a minimal homeomorphism. As slow recurrence implies the CE-condition, we now modify the construction from Section~\ref{Section: Absorbing Family} to generate the kneading sequence of an S-unimodal map $f$ such that $f|_{\omega(c)}$ is conjugate to an odometer and $\mathcal{K}(f)$ is slowly recurrent.

\begin{lemma}
    Let $H_0C$ and $E_0C$ be two shift-maximal kneading sequences that are the same length and disagree in only one position. When $H_0$ has even parity set $\Delta = 1$, and set $\Delta = 0$ if $H_0$ has odd parity; let $\Delta' = 1-\Delta$. Consider $j\geq 2$ and define $H_1C = H_0(\Delta E_0)^{j-1}C$ and $E_1C = H_0\Delta'E_0(\Delta E_0)^{j-2}C$. Then $H_1C$ and $E_1C$ are shift-maximal.
\end{lemma}
\begin{proof}
    The proof is analagous to the proofs for shift-maximality in \cite{Alvin_strange_star} and in Section~\ref{Section: Absorbing Family}.
\end{proof}

\begin{proposition}\label{Proposition: SR Pattern generation}
    Let $(H_0,E_0,\gamma)$ be a kneading odometer triple (but here we relax the restriction on $\gamma$ so that each $j_i\geq 2$). Define $H_\infty = \lim_{n\to\infty} H_n C$ where $$H_nC = H_{n-1}(\Delta E_{n-1})^{j_n-1}C \text{ and } E_nC = H_{n-1}\Delta'E_{n-1}(\Delta E_{n-1})^{j_n-2}C.$$ Then $H_\infty$ is shift-maximal and belongs to an S-unimodal map $f$ such that $f|_{\omega(c)}$ is conjugate to the odometer with $\gamma' = (|H_0|+1,j_1,j_2,\ldots)$.  
\end{proposition}
\begin{proof}
    The proof is analagous to the proof for Theorem~\ref{Theorem: infinite strange star}.
\end{proof}

We now use Proposition~\ref{Proposition: SR Pattern generation} to construct a slowly recurrent kneading sequence belonging to a unimodal map $f$ with $f|_{\omega(c)}$ conjugate to an odometer.
\begin{example}\label{Example: slowly recurrent odometer}
    Let $H_0C = 1001C$, $E_0C = 1011C$, and $\gamma = (2^2,2^3,\ldots, 2^{i+1},\ldots )$. Construct $H_\infty$ as in Proposition~\ref{Proposition: SR Pattern generation}. Then $H_\infty$ is shift-maximal and belongs to an S-unimodal map $f$ such that $f|_{\omega(c)}$ is conjugate to the odometer generated by $\gamma'=(5,2^2,2^3,\ldots)$. We write the start of the kneading sequence $H_\infty$ here so that some patterns can be easily understood. We indicate cutting times with a period and co-cutting times with a prime.

    \begin{align*}
        H_\infty = & 1.0.0.1'1.1'0'1.1'1.1'0'1.1'1.1'0'1.1'1.1'0'0'10.1'0'1.1'1.1'0'1.1'1.1'0'1.1'1.\\
         & 1'0'0'10.1'0'1.1'1.1'0'1.1'1.1'0'1.1'1.1'0'0'10.1'0'1.1'1.1'0'1.1'1.1'0'1.1'1. \\
         & 1'0'0'10.1'0'1.1'1.1'0'1.1'1.1'0'1.1'1.1'0'0'10.1'0'1.1'1.1'0'1.1'1.1'0'1.1'1. \\
         & 1'0'0'10.1'0'1.1'1.1'0'1.1'1.1'0'1.1'1.1'0'0'10.1'0'1.1'1.1'0'1.1'1.1'0'1.1'1. \\
         & 1'0'0'11'101'11'101'11'101'10.1'0'0'10.1'0'1.1'1.1'0'1.1'1.1'0'1.1'1. \cdots
         %& 1'0'0'10.1'0'1.1'1.1'0'1.1'1.1'0'1.1'1.1'0'0'10.1'0'1.1'1.1'0'1.1'1.1'0'1.1'1.
    \end{align*}
    First note that if $|H_{n-1}C|< l \leq |H_nC|$, then $\RR(m)\geq l$ if and only if $m$ is a multiple of $|H_{n+1}C|$. Further, each multiple of $|H_{n+1}C|$ is a return time, and thus all other return times will have $\RR(M_i) < l$. 
    
    For each $n\geq 0$ set $l_n = |H_nC|$. We may calculate an upper bound on $$S_n = \limsup_{i\to \infty} \frac{\sum_{j=1}^i \begin{cases}
        \mathcal{R}(M_j) & \text{if } \mathcal{R}(M_j) \geq l_n \\ 
        0 & \text{otherwise}
    \end{cases}}{M_i}$$ by considering the sequence of return times of the form $|H_kC|$ for each $k\geq 1$.

    Fix $n=0$ so that $l_0 = |H_0C| = 5$ and first consider $M_i = |H_1C|$. Then the ratio of the summation of long agreements of the return times through $M_i$ (which we will simply call the summation ratio) is $$\frac{|H_0C|}{|H_1C|} = \frac{5}{5\cdot 2^2} = \frac{1}{2^2}.$$
    Now consider $M_i = |H_2C|$. Then the summation ratio is $$\frac{(2^3-1)|H_0C| + |H_1C|}{|H_2C|} = \frac{11}{2^5}\leq \frac{2^4}{2^5} = \frac{1}{2}.$$
    When $M_i = |H_3C|$, the summation ratio is $$\frac{(2^4-1)\left((2^3-1)|H_0C|+|H_1C| \right) + |H_2C|}{|H_3C|}\leq \frac{2^4(2^3\cdot 5 + 2^2\cdot 5)+2^5\cdot 5}{2^9\cdot 5} = \frac{14}{32} \leq \frac{2^4}{2^5} = \frac{1}{2}.$$
    By continuing in this way, we see that $S_0 \leq \frac{1}{2}.$

    Now fix $n=1$ (so that $l_1=|H_1C|=2^2\cdot 5$) and first consider $M_i = |H_2C|$. Then the ratio of the summation of long agreements of the return times through $M_i$ is $$\frac{|H_1C|}{|H_2C|} = \frac{1}{2^3}.$$
    Now for $M_i = |H_3C|$, the summation ratio is $$\frac{(2^4-1)|H_1C| + |H_2C|}{|H_3C|} = \frac{(2^4-1)\cdot 2^2\cdot 5 + 2^5\cdot 5}{2^9\cdot 5} \leq \frac{2^4\cdot 2^2 + 2^5}{2^9} = \frac{3}{2^4} \leq \frac{2^2}{2^4} = \frac{1}{2^2}.$$
    For $M_i= |H_4C|$, the summation ratio is $$\frac{(2^5-1)((2^4-1)|H_1C| + |H_2C|) + |H_3C|}{|H_4C|} \leq \frac{2^9 +2^8+2^7}{2^{12}} = \frac{7\cdot 2^7}{2^{12}}\leq \frac{1}{2^2}.$$
    Iterating the same estimate gives
\[
S_n
\leq
\sum_{r=n+1}^{\infty}\frac{1}{2^r}
=
\frac{1}{2^n}.
\]
After shifting the indexing convention for $l_n=|H_nC|$, this gives the desired
bound $S_n\to 0$ as $n\to\infty$. Hence $\mathcal{K}(f)$ is slowly recurrent.
\end{example}

We note that by varying the initial shift-maximal sequences $H_0C$ and $E_0C$ in the previous example, we can generate a whole family of unimodal maps that are slowly recurrent and have a minimal homeomorphism on their $\omega$-limit sets. Additionally, we can vary the generating sequence $\gamma = (j_1,j_2,\ldots)$ so that the embedded odometer has a different underlying period structure provided that $j_i\to \infty$ quickly enough. Because any sequence $\gamma = (j_1,j_2,\ldots)$ can be rewritten as $\gamma'= (j_1\cdots j_{k_1}, j_{k_1+1}\cdots j_{k_2},j_{k_2+1}\cdots j_{k_3},\ldots)$, we conclude that we can generate examples of slowly recurrent S-unimodal maps with any embedded odometer by following the construction in Proposition~\ref{Proposition: SR Pattern generation}. 
We emphasize that as we used a different strange star product in this construction than what was used in Section~\ref{Section: Absorbing Family}, the maps constructed in this way are not persistently recurrent and thus do not belong to maps with wild Cantor attractors.

\begin{remark}
    Example~\ref{Example: slowly recurrent odometer} is an example of a map $f$ with a regularly recurrent turning point such that $\mathcal{K}(f)$ is slowly recurrent. It is also possible for a kneading sequence $\mathcal{K}(f)$ to be linearly recurrent with a regularly recurrent turning point and thus $\mathcal{K}(f)$ is not slowly recurrent. Hence there is no clear relationship between slow recurrence and regular recurrence (and thus uniform recurrence).
\end{remark}

%\Jernej{
%\begin{question}
   %Is Example~\ref{Example: reg rec and longbranched} slowly recurrent? 
   %Is Example~\ref{Example: reg rec and longbranched} CE? What is the relationship with critically finite (which is equivalent to CE)? %\Lori{Example 26 is not slowly recurrent. No kneading sequence from a left-proper substitution will be slowly recurrent.}
%\end{question}
%}

%
%
%
%
%
%
%
%
%
%
%
\section{Longbranchedness}\label{section: Longbranchdness}
It is well-known that if $Q(k) \to \infty$, then $\widetilde{Q}(k) \to \infty$. Conversely, if $\widetilde{Q}(k)$ is bounded, then so is $Q(k)$.

\begin{definition}
    A unimodal map $f$ is called \emph{longbranched} if there exist some $\delta >0$ such that $|D_n|>\delta$ for all levels of the Hofbauer tower; equivalently, if $Q(k)$ is bounded.
\end{definition}

Observe that if $Q(k)$ is bounded and $c$ is recurrent (not periodic), then $\widetilde{Q}(k)$ must be unbounded. We now prove that if $f$ is longbranched, then $c$ is not persistently recurrent.

	\begin{theorem}\label{thm:longbranched not persistently recurrent}
		If $Q(k)$ is bounded and $c$ is infinitely recurrent, then $c$ is not persistently recurrent.
	\end{theorem}
        \begin{proof}
            Suppose that $c$ is infinite recurrent and that $Q(k)$ is bounded. Let $(n_i)$ be a sequence of co-cutting times such that $\RR(n_i)\to\infty$. Fix $\epsilon > 0$ and let $U = (c-\epsilon, c+\epsilon)$. Since $Q$ is bounded, let $B=\sup_k Q(k)<\infty$. Choose $\epsilon>0$ small
            enough so that the corresponding $M$ from Lemma~\ref{Lemma: close to c}
            satisfies \(M>S_B\).
            Then $M>S_{Q(k)}$ for all $k\in \N$.
             Because $\RR(n_i)\to \infty$ there exists some $I\in \N$ such that for all $i\geq I$, $\RR(n_i) > M$. Observe that $\RR(n_i - s(n_i)) < M$ for all $i\in\N$, and thus $c_{n_i-s(n_i)}\notin U$ for any $i$. If $c$ is persistently recurrent, then it must follow that for all $i$ large enough, $c_{n_i - \tilde{s}(n_i)}\in U$. Observe that because $n_i$ is a co-cutting time, then $n_i - \tilde{s}(n_i) = S_t$ for some $t\in \N$ (see Lemma~\ref{Lemma: helpful cutting facts}). Thus, $\RR(n_i - \tilde{s}(n_i)) = \RR(S_t) = S_{t+1}-S_t = S_{Q(t)} < M$. We thus conclude that $c_{n_i - \tilde{s}(n_i)}\notin U$. Therefore $c$ is not persistently recurrent.
        \end{proof}

In the introduction of \cite{FPEP}, the authors state that when $c$ is infinitely recurrent and longbranched, then $\mathcal{F}\neq \mathcal{E}$ can be empty, countable or uncountable. Recall that $\mathcal{F} = \mathcal{E}$ if and only if $c$ is persistently recurrent. Thus $\mathcal{F}\setminus \mathcal{E}=\emptyset$ is not possible as the following corollary shows. %\Jernej{Wasn't there an example in some Henk's paper where $c$ is infinitely recurrent and longbranched but $\mathcal{F}\setminus \mathcal{E}=\emptyset$?}.    

	\begin{corollary}
		If $Q(k)$ is bounded and $c$ is infinite recurrent, then $\mathcal{F}\neq \mathcal{E}$. 
	\end{corollary}

We now present two examples of unimodal maps $f$ such that $Q(k)$ is bounded and $\widetilde{Q}(k)\to \infty$. We note that in both cases $c$ is recurrent but not uniformly recurrent.

% \begin{lemma}
% 	Let $c$ be infinite recurrent and persistently recurrent. If $Q(k)$ is bounded, then $\widetilde{Q}(k)$ is unbounded.
% \end{lemma}
% \begin{proof}

% 	If $Q(k)$ is bounded, then $n-s(n)$ is bounded. Fix $\epsilon >0$ such that $|c-c_t|>\epsilon$ for all $t \leq \max\{n-s(n): n\in \N\}$.  Recall that $\widetilde{D}_n = [c_{n-\widetilde{s}(n)}; c_{n-s(n)}]$ is the image of the maximal monotone lap containing $c_1$ under $f^{n-1}$. Consider the sequence $(n_i)$ defined by $|c_{n_i}-c|<\epsilon$. As $c$ is persistently recurrent, there exists an $N$ such that for all $n_i\geq N$ we have $\widetilde{D}_{n_i}\not\supseteq B(c,\epsilon)$ and thus $|c_{n_i - \tilde{s}(n_i)}-c| < \epsilon$. Because $\epsilon$ was chosen arbitrarily, we must have that some subsequence $ c_{n_{i_j} - \widetilde{s}(n_{i_j})} \to c$ as $j \to \infty$. Hence $n - \widetilde{s}(n)$ is unbounded, which means that $\widetilde{Q}(k)$ is unbounded.
% \end{proof}

\begin{example}\label{Example: co-kneading goes to infinity not uniformly recurrent 1}
	There exists a unimodal map $f$ such that $Q(k)$ is bounded and $\widetilde{Q}(k) \to \infty$.  Further, in this example, $c$ is not uniformly recurrent.
	
	Let \begin{align*} \mathcal{K}(f) = &~ 1.0.0.1'1.0.1'1.0.0.11.1'1.0.0.11. 0.10'1.0.0. \\
		&~ 11. 0. 11. 1'1. 0.0.11. 0. 11. 0.0.11. 11. 1'1.0.0.\\
		&~ 11. 0. 11. 0.0.11.11.0.0.11.0.101.1'1.0.0.11. \cdots\\	
	\end{align*} We note that $Q(k) \leq 2$ for all $k\in \N$. This kneading sequence is generated by concatenating the words $0$, $11$, and $101$ so that $Q(k)$ remains bounded, but $\widetilde{Q}(k) \to \infty$. Here $\widetilde{Q}(1) = 2, \widetilde{Q}(2) = 4, \widetilde{Q}(3) = 5, \widetilde{Q}(4) =  6, \widetilde{Q}(5) = 10, \widetilde{Q}(6) = 15, \cdots$. In this construction, the sequence after each co-cut agrees with the initial kneading sequence until the `next' block of the form `$.0.0.$' which is then replaced with `.11.'  We note that in the kneading sequence, arbitrarily long blocks of $1$s will appear, and therefore $c$ will not be uniformly recurrent. %As such, it follows that $c$ is not persistently recurrent.  Thus, this example also demonstrates that $\widetilde{Q}(k)\to \infty$ does not imply that $c$ is persistently recurrent. 
\end{example}

\begin{example}\label{Example: co-kneading goes to infinity not uniformly recurrent 2}
 Let $f$ be the unimodal map with 
 \begin{align*}
 	\mathcal{K}(f) =&~ 1. 0. 0. 1'1. 1'0'1. 1'0'0'10. 0. 10'1. 0. 0. 11. 100'10. 0. 11. 101.  10010. 1'1.\\ &~ 0. 0. 11. 101. 10010.0.101.0.1'1.0.0.11.101. 10010.0.101.0.0.11. \\ &~ 10010.1'1.0.0.11.101.10010. 0.101.0.0.11.10010.0.11.101.10010.\\
    &~11.0.1'1.0.0.11. 101.10010.0.101.0.0.11.10010.0.11.101.10010.11.\\
    &~ 0.0.11.101.10010.1'1.\cdots
 \end{align*} 
 
 Here the set of concatenated words is $W = \{0, 11, 101, 10010\}$, and $\mathcal{K}(f)$ is constructed in such a way that eventually after every co-cutting time, the sequence agrees with the kneading sequence up to the `next' $10010.0$ or $0.0$, and the second $0$ word is replaced with $11$.  In this way, $\widetilde{Q}(k) \to \infty$ while $Q(k)$ remains bounded.  Additionally, because the word $(0.11)^n$ will appear in $\mathcal{K}(f)$ for arbitrarily large $n$ (each occurrence of $11.0.11.0.0.$ will lead to a later occurrence of $11.0.11.0.11.0.0$), this example is not uniformly recurrent.
\end{example}

\subsection{Sturmian kneading sequences}\label{subsection: Sturmian}

We recall that Examples~\ref{Example: co-kneading goes to infinity not uniformly recurrent 1} and \ref{Example: co-kneading goes to infinity not uniformly recurrent 2} were both longbranched with $\widetilde{Q}\to \infty$, and neither example was uniformly recurrent. In this subsection we will define kneading sequences that appear in the tent map family, are longbranched, have $\widetilde{Q}(k)\to\infty$, and whose turning points are uniformly recurrent.

\begin{comment}
\begin{question}
     If $Q$ is bounded and $\widetilde{Q}\to \infty$, is it possible for $c$ to be uniformly recurrent?
\end{question}
\end{comment}

 Sturmian kneading sequences can be defined through irrational rotations on the circle as follows (see also \cite{Outershark}).
Let $f:I \to I$ be a symmetric unimodal map, i.e.,
given the involution $\tilde x:= 1-x$, we assume that
$f(\tilde x) = f(x)$ for every $x$.
This means that the critical point $c = \frac12$, and by an appropriate scaling, we can assume that $f(c) = 1$.
For instance, $f_a(x) = 1-a(x-\frac12)^2$ with $a \in (0,4]$ is the logistic family in this scaling.

There is a natural way to turn $f$ into an increasing symmetric Lorenz map  $\varphi:I \to I$ by flipping the right half of the graph vertically 
around $c = \frac12$, giving:
$$
\varphi(x) = \begin{cases}
              f(x) & \text{ if } x \in [0,c],\\
              \widetilde{f(x)} & \text{ if } x \in (c,1].
             \end{cases}
$$

Now we turn $\varphi$ into a proper circle endomorphism (with unique rotation number independent of $x \in \mathbb{S}^1$)
by setting:
$$
\bar\varphi(x) = \begin{cases}
               \varphi(1)=\widetilde{f(1)}, & x \in [0,a];
               \text{ where } a < c \text{ is such that } \varphi(a)=\varphi(1),\\
               \varphi(x), & \text{otherwise.}
              \end{cases}             
$$
Also let $b > c$ be such that $\varphi(b) = a$, see Figure~\ref{fig:barLorenz}.

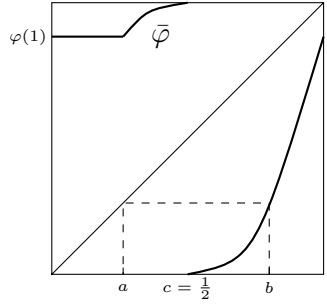
\begin{figure}[ht]
\begin{center}
\begin{tikzpicture}[scale=0.9]
\draw[-, draw=black] (0,0) -- (4,0) -- (4,4) -- (0,4) -- (0,0) -- (4,4); 
% \draw[-, draw=black] (6,0) -- (10,0) -- (10,4) -- (6,4) -- (6,0);
 \draw[-, thick] (0,3.5) -- (1.05, 3.5);
 \draw[-, thick] (1.05,3.5) .. controls (1.4, 3.9) .. (2,4); 
\draw[-] (1.05,0) -- (1.05,0.1); 
\draw[-] (3.2,0) -- (3.2,0.1); 
\draw[dashed] (3.2,0) -- (3.2,1.05)--(1.05,1.05)--(1.05,0); 
 \draw[-, thick] (4,3.5) .. controls (3, 0.2) .. (2,0);
%\draw[-, thick] (6,0.5) .. controls (7, 3.8) .. (8,4); \draw[-, thick] (10,3.5) .. controls (9, 0.2) .. (8,0); 
\node at (2,-0.2) {\tiny $c=\frac12$}; \node at (1.05,-0.2) {\tiny $a$}; \node at (-0.35,3.5) {\tiny $\varphi(1)$};  
\node at (1.6, 3.5){$\bar \varphi$}; \node at (3.2,-0.2) {\tiny $b$};
\end{tikzpicture}
\caption{\label{fig:barLorenz} A stunted symmetric Lorenz map $\bar{\varphi}$ as a circle endomorphism.}
\end{center}
\end{figure} 

%The circle endomorphism $\bar{\varphi}$ obtained from $\varphi$ was already studied in the last section of \cite{bruinThesis}.
The following is a corrected version of Proposition 2 from \cite{Outershark}, which also appears in the errata for \cite{BruinBook}.

\begin{proposition}\label{prop:rotnumber}
Assume that $f$ is a unimodal map with cutting times $\{ S_j\}_{j \geq 0}$.
Let $b > c$ be such that $\bar{\varphi}(b)=a$, see Figure~\ref{fig:barLorenz}.
Then the rotation number of the corresponding $\bar\varphi$ equals
 $$
 \alpha = \begin{cases}
         \frac{k}{S_k} \in [\frac12, 1] \cap \mathbb{Q} & \text{ if $k$ is minimal such that } 
         f^{S_k}(c) \in (\tilde b,b),\\[1mm]
         \lim_{k \to \infty} \frac{k}{S_k} \in [\frac12,1] & \text{ if no such $k$ exists.}
        \end{cases}
 $$
In the latter case, the kneading map $Q(j) \leq 1$ for all $j \in \N$,
and if $\alpha \notin \mathbb{Q}$, then $f:\omega(c) \to \omega(c)$ is a non-invertible minimal surjection.
\end{proposition}

Namely, $f:\omega(c) \to \omega(c)$ is not a homeomorphism. The fault in reasoning in \cite{Outershark} arises in the last paragraph of the proof of Proposition 2 from \cite{Outershark} where the authors assert: 
\begin{quote}
We will show that $f:\omega_{\varphi}(c) \to \omega_f(c)$ is in fact a homeomorphism, from which it follows that  $f:\omega_{f}(c)\to\omega_f(c)$ is also a homeomorphism.
\end{quote}

However, from the fact that $f:\omega_{\varphi}(c) \to \omega_f(c)$ is a homeomorphism it need not follow that $f:\omega_{f}(c)\to\omega_f(c)$ is also a homeomorphism. Even more, this does not hold as Theorem~\ref{Theorem:Alvin_homeo} is not satisfied.

Furthermore, when the rotation number in Proposition~\ref{prop:rotnumber}
is $\alpha = \lim_{k\to\infty}k/S_k$, then $(\omega(c), f)$ represents a Sturmian shift.
Even more, every irrational rotation number (hence every Sturmian shift) can be realized this way as it is argued after the proof of Proposition~2 in \cite{Outershark}. Let us recall, correct and extend the explanation of this procedure since it will be important to understand it when determining the cutting and co-cutting times of the corresponding unimodal maps.

We can split any sequence $e \in \{ 0,1\}^\N$
into maximal pieces (up to the last symbol) that coincide with a prefix of $\nu$.
To this end, define
\begin{equation}\label{eq:rho-function}
\rho:\N \to \N, \quad \rho(n)=\max\{k>n : e_{n+1}e_{n+2}\cdots e_{k-1}
\text{ is a prefix of }\nu\}.
\end{equation}

That is, the function $\rho$ depends on $e$ and $\nu$, but we will suppress this dependence.
When we apply this for $e = \nu$, we obtain
$$
S_0 = 1, \quad S_{k+1} = \rho(S_k),
$$
or in other words $S_k = \rho^k(1)$ for $e=\nu$ and $k \geq 0$.

Let \(1-\alpha=[0;a_1,a_2,\ldots]\in(0,1/2]\setminus\mathbb Q\), with convergents \(p_i/q_i\). Since \(1-\alpha\le 1/2\), we have \(a_1\ge2\).
%Let $1-\alpha = [0;a_1,a_2,a_3,\dots]$ be the continued fraction expansion of $1-\alpha\in[\frac{1}{2},1]\setminus \mathbb{Q}$, with convergents $\frac{p_i}{q_i}$. Since $\alpha\in [1/2,1]$ we obtain $a_1=2$. 
In the rest of this paragraph we refer the reader to \cite[Chapter I, II]{Khinchin}.
Due to the property that the convergents are best approximates for the irrational rotation $R_\alpha$, it holds that the denominators $q_i$ are the times of
closest returns of any point $x \in \mathbb{S}^1$ to itself. Furthermore, these closest returns occur in an alternating fashion on the left and on the right of $x$.
Therefore, if we assume that $R_\alpha^{q_i}(x)$ is to the right of $x$, and let interval $A_{q_i} = [x,R_\alpha^{q_i}(x)]$,
then the first iterate $k$ such that $x\in R_\alpha^k(A_{q_i})$ is $k = q_{i+1}$ and $R_\alpha^{q_{i+1}}(x)$ is to the left of $x$.

For the map $\bar\varphi$, the closest returns on the left indeed accumulate on $c$, but the right neighborhood
$[c,b)$ is the preimage of the plateau $[0,a)$ and no further iterates of $c$ enter that region.
Instead, returns on the left accumulate on $b$.

Translating this back to the unimodal map $f$ with kneading sequence $\nu = \nu_1\nu_2\nu_3\dots$, 
the closest returns on the left correspond to
closest returns at co-cutting times Indeed, recall that in the irrational case of Proposition~\ref{prop:rotnumber}, no cutting time \(S_j\) satisfies \(f^{S_j}(c)\in(\tilde b,b)\).
Let $q_i$ be such a co-cutting time. Then it follows from the definition of \(\rho\) in \eqref{eq:rho-function},
together with the description of closest returns for irrational rotations,
that the co-cutting times are generated from the denominators \(q_i\) of the
convergents of \(1-\alpha=[0;a_1,a_2,\ldots]\) as follows.
Recall that the convergents satisfy the recurrence
\[
q_{i+2} = a_{i+1}q_{i+1} + q_i.
\]
If \(q_i\) is a co-cutting time, then the successive co-cutting times obtained
from \(q_i\) are given by the iterates of \(\rho\):
\begin{equation}\label{eq:rhoa}
\rho^a(q_i) = q_i + a q_{i+1}, \qquad 1 \le a \le a_{i+1}.
\end{equation}
In particular, the last co-cutting time in this block satisfies
\begin{equation}\label{eq:rhoa+1}
\rho^{a_{i+1}}(q_i) = q_i + a_{i+1}q_{i+1} = q_{i+2}.
\end{equation}

Note that in \cite{Outershark} the authors started with $\alpha = [0;a_1,a_2,a_3,\dots]\in [0,\frac{1}{2}]$. However, since then necessarily $a_1=1$, it would imply that $q_0=0$ and $q_1=1$; by using the above procedure this would not produce a valid kneading sequence. To correct this statement the procedure needs to start with $\alpha=\max\{\tilde{\alpha},1-\tilde{\alpha}\}$, where we denote $\alpha = [0;a_1,a_2,a_3,\dots]\in [0,1]$.

\begin{proposition}\label{prop:Sturmian}
Let $\nu$ be a Sturmian kneading sequence. Then $\widetilde{Q}(k)\to \infty$ and $c$ is uniformly recurrent.
\end{proposition}

\begin{proof}
Let \(q_i\) be a co-cutting time. By the construction above, the successive
co-cutting times generated from \(q_i\) are
\[
\rho^a(q_i)=q_i+a q_{i+1},
\qquad 1\le a\le a_{i+1},
\]
and
\[
\rho^{a_{i+1}}(q_i)=q_{i+2}.
\]
Hence, within this block, the difference between consecutive co-cutting times is
\[
\rho^{a+1}(q_i)-\rho^a(q_i)=q_{i+1},
\qquad 1\le a<a_{i+1}.
\]
By the definition of the co-kneading map,
\[
S_{\widetilde Q(l)}
=
\widetilde S_l-\widetilde S_{l-1}.
\]
Thus each difference between consecutive co-cutting times is a cutting time.
In the block generated from \(q_i\), this cutting time is \(q_{i+1}\). Since
\(q_{i+1}\to\infty\), and since \(S_m\to\infty\) with \(m\), the corresponding
indices \(\widetilde Q(l)\) also tend to infinity. Therefore
\(
\widetilde Q(l)\to\infty .
\)
Finally, Sturmian subshifts are minimal. Hence every Sturmian sequence is
uniformly recurrent, so every initial block of \(\nu\) appears in \(\nu\) with
bounded gaps. By the symbolic criterion for uniform recurrence of the turning
point, \(c\) is uniformly recurrent.
\end{proof}

For example, if $\alpha=\sqrt{2}-1$, then $\alpha= [0;2,2,2,\ldots]$, i.e. $a_i=2$ for all $i\in \N$. Therefore, the denominators of convergents $q_i$s are the Pell numbers $1,2,5,12,29,70,169,\dots$. 
Since in general, $\kappa\geq 1$ it follows that $2$ is a cutting time and thus $q_2=5=\widetilde{S}_2$ is a co-cutting time. 

Since \(q_2=5\) is a co-cutting time, equations~\eqref{eq:rhoa} and~\eqref{eq:rhoa+1} give
\[
\rho(5)=\widetilde S_3=17
\qquad\text{and}\qquad
\rho^2(5)=\widetilde S_4=29=q_4.
\]
Furthermore,
\[
\rho(29)=\widetilde S_5=99
\qquad\text{and}\qquad
\rho^2(29)=\widetilde S_6=169=q_6.
\]
Since \(\widetilde S_1=3\), we obtain
\[
\widetilde S_2-\widetilde S_1=2,\qquad
\widetilde S_3-\widetilde S_2=\widetilde S_4-\widetilde S_3=12,
\]
and
\[
\widetilde S_5-\widetilde S_4=\widetilde S_6-\widetilde S_5=70.
\]

\begin{comment}
\Jernej{Since \(q_2=5\) is a co-cutting time, equations \eqref{eq:rhoa} and
\eqref{eq:rhoa+1} give
\[
\rho(5)=17,\qquad \rho^2(5)=29=q_4.
\]}
\sout{Thus using equations \eqref{eq:rhoa} and \eqref{eq:rhoa+1} we obtain}
\[
\sout{\rho^1(2)=\tilde{S}_3=17} \text{ and } \sout{\rho^2(2)=\tilde{S}_4=29=q_4.}
\]
Furthermore, 
\[\rho^1(29)=\tilde{S}_5=99 \text{ and } \rho^2(29)=\tilde{S}_6=169=q_{6}.
\]
Since $\tilde S_1=3$ we get 
\[\tilde{S}_2-\tilde{S}_1=2, \tilde{S}_3-\tilde{S}_2=\tilde{S}_4-\tilde{S}_3=12 \text{ and } \tilde{S}_5-\tilde{S}_4=\tilde{S}_6-\tilde{S}_5=70.
\]
\end{comment}

Therefore, we obtain
$$
\nu = 1\boldsymbol{0}.1'1.\boldsymbol{1}'1.0.11.11.\boldsymbol{0}.11.11.1'1.0.11.11.0.11.11.\boldsymbol{1}'1.0\dots
$$
where dots indicate cutting times and primes co-cutting times. The bold symbols indicate the positions $q_i$; they are alternatingly cutting and co-cutting times as argued above. %Observe that $\tilde{Q}(i)\to \infty$.
In fact, if we focus only on the co-cutting times, because $q_i$ are denominators of the convergents for an irrational number it holds that
$$
 \nu_{q_{i+1}-q_i+1} \dots \nu_{q_{i+1}-1}\nu_{q_{i+1}} =
 \nu_1 \dots \nu_{q_i-1} \nu'_{q_i} 
 %\text{ or }   \nu_1 \dots \nu_{q_i-1} \nu'_{q_i} \qquad
 \text{ for each odd } i \in \N,
$$
and therefore $c$ has two limit itineraries $\lim_{x \nearrow c} i(x) = 0\nu$ and $\lim_{x \searrow c} i(x) =1\nu$. Thus the shift space generated by \(\nu\) does not satisfy the two-preimage condition in Theorem~\ref{Theorem:Alvin_homeo}. Hence \(f|_{\omega(c)}\) is not a minimal homeomorphism.
%\Jernej{Also check the last displayed equation! even vs. odd}
%\Lori{In the previous statement, it says `for all i' and `for all even i'. I am also confused by the upper and lower itineraries here; should we mention that only $1\nu$ appears in the shift space generated by the kneading sequence?}

%Now write that even all Sturmian in unimodal maps are represented like that and all of them appear in the tent map family.

%next insert a claim that turning point is long branched and uniformly recurrent

%next insert the example with Pell numbers 

%\Lori{We need to give Henk credit for his corrected version of the proof that he put in his errata. The question remains whether it is possible for a longbranched map with infinitely recurrent turning point to have a minimal homeomorphism.}

The final statement in Proposition 2 from \cite{Outershark} aimed to address the question whether there exists a longbrached unimodal map $f$ such that $f:\omega_{f}(c) \to \omega_f(c)$ is a minimal homeomorphism. Thus, Question~\ref{Question: longbranched} still remains open.

\begin{comment}
\begin{question}
Let $f$ be a longbranched (infinitely) recurrent unimodal map. Is it possible that $f:\omega_{f}(c) \to \omega_f(c)$ is a minimal homeomorphism?
\end{question}
\end{comment}

% Observe that in Examples~\ref{Example: co-kneading goes to infinity not uniformly recurrent 1} and \ref{Example: co-kneading goes to infinity not uniformly recurrent 2}, we have that $f$ is longbranched, $\widetilde{Q}(k)$ is unbounded, $c$ is not uniformly recurrent, and we get arbitrarily long repeated blocks in the kneading sequence. This leads to the following question.

% \begin{question}
%     If $Q(k)$ is bounded and $\widetilde{Q}$ is unbounded, will $c$ be uniformly recurrent if and only if there is no word $w = w_1w_2w_3\cdots w_j$ such that $w^n$ appears in $\mathcal{K}(f)$ for arbitrarily large $n$?
% \end{question}
% %
%
%
%
%
%
\subsection{Existence of a longbranched map with regularly recurrent turning point}
We have previously mentioned that there exist unimodal maps that are longbranched and whose turning points are uniformly recurrent. We now show that it is also possible for a longbranched unimodal map to have a regularly recurrent turning point.

\begin{example}\label{Example: reg rec and longbranched}
    Let $A_1 = 1.0.0.11.1$ and $B_1 = 1.0.101.1$. For all $n\geq 1$, define $A_{n+1} = A_nA_nB_n$ and $B_{n+1} = A_nB_nB_n$. 
\end{example}

\begin{lemma} 
    $A_k$ and $B_k$ have even parity for all $k$. Additionally, $A_k$ and $B_k$ differ in only the $|A_{k-1}|+|A_{k-2}| + \cdots + |A_1| + 3$ and $|A_{k-1}|+|A_{k-2}| + \cdots + |A_1| + 4$ positions and $A_k \succ B_k$ for all $k \geq 1$.
\end{lemma}
\begin{proof}
    Both $A_1$ and $B_1$ clearly have even parity. Observe that for every $k\geq 1$, $A_{k+1} = A_k A_kB_k$ and $B_{k+1} = A_kB_kB_k$. Hence it inductively holds that both $A_k$ and $B_k$ have even parity for every $k\geq 1$.

    Note that $A_1$ and $B_1$ differ only in the third and fourth positions. Because $A_2 = A_1A_1B_1$ and $B_2 = A_1B_1B_1$, $A_2$ and $B_2$ differ in the third and fourth positions of the central block. This occurs in positions $|A_1| + 3$ and $|A_1| + 4$. Suppose that $A_k$ and $B_k$ differ in only the $|A_{k-1}|+|A_{k-2}| + \cdots |A_1| + 3$ and $|A_{k-1}|+|A_{k-2}| + \cdots + |A_1| + 4$ positions.  Then $A_{k+1} = A_kA_kB_k$ and $B_{k+1}=A_kB_kB_k$ differ only in the central block in the position where $A_k$ and $B_k$ differ. This is precisely the positions $|A_k| + |A_{k-1}|+|A_{k-2}| + \cdots |A_1| + 3$ and $|A_k| + |A_{k-1}|+|A_{k-2}| + \cdots |A_1| + 4$. By induction, the pattern holds.

    It is easy to check that $A_1 \succ B_1$. This is because $10$ has odd parity, so $100 \succ 101$. By combining the facts that $A_k$ and $B_k$ differ first in the $|A_{k-1}|+|A_{k-2}| + \cdots +|A_1| + 3$ position and that each $A_j$ has even parity, observe that $A_k$ and $B_k$ agree on the entire block $A_kA_{k-1}A_{k-2}\cdots A_1~10$, which has odd parity, and disagree in the next position. Since the next position of $A_k$ is a $0$, it follows that $A_k \succ B_k$.
\end{proof}

\begin{lemma} $A_nC$ is shift-maximal for all $n\in \N$.
\end{lemma}
\begin{proof}
Observe that $A_1C$ is shift-maximal, $A_1 \succ \sigma^j(A_1)$ for all $1\leq j\leq |A_1|-2$, and $A_1C\succ \sigma^j(B_1C)$ for all $0\leq j \leq |A_1|$.

Suppose that for some $k\geq 1$ we have that $A_kC$ is shift-maximal, $A_k \succ \sigma^j(A_k)$ for all $1\leq j\leq |A_k|-2$, and $A_kC \succ \sigma^j(B_kC)$ for all $0\leq j \leq |A_k|$.  Look at $A_{k+1}C = A_kA_kB_kC$. Because $A_k \succ \sigma^j(A_k)$ for all $1\leq j\leq |A_k|-2$, it follows that $A_{k+1}C \succeq \sigma^j(A_{k+1}C)$ for $0\leq j\leq |A_k|-2$ and $|A_k|+1 \leq j \leq 2|A_k|-2$. Additionally, because $\sigma^{|A_k|-1}(A_{k+1}C)$ begins with $11 \prec 10$, it follows that $A_{k+1}C \succeq \sigma^j(A_{k+1}C)$ for $j = |A_k|-1$. Because $A_kA_k \succ A_kB_k$, it follows that $A_{k+1}C \succeq \sigma^j(A_{k+1}C)$ for $j = |A_k|$. Since $\sigma^{2|A_k|-1}(A_{k+1}C)$ begins with $11$, we see that $A_{k+1}C\succ \sigma^j(A_{k+1}C)$ for $j = 2|A_k|-1$. Lastly, because $A_k 1 \succ A_kC \succeq \sigma^j(B_kC)$ for all $0\leq j \leq |A_k|$, it follows that $A_{k+1}C \succ \sigma^j(A_{k+1}C)$ for all $2|A_k| \leq j \leq |A_{k+1}|$.

It thus follows that $A_{k+1}C$ is shift-maximal. Hence, by induction, $A_nC$ is shift-maximal for all $n\in \N$.
\end{proof}

\begin{theorem}
    Let $A_n$ and $B_n$ be defined as in Example~\ref{Example: reg rec and longbranched}. Then the sequence $w = \lim_{n\to \infty}A_n$ is the kneading sequence for a unimodal map that is longbranched and has regularly recurrent turning point. 
\end{theorem}
\begin{proof}
    First note that $w = \lim_{n\to \infty} A_nC$. Since $A_nC$ is shift-maximal for all $n\in \N$, we conclude that $w$ is shift-maximal. Hence $w$ is the kneading sequence for some unimodal map. By \cite[Corollary 4.6]{Alvin_uniformly_recurrent}, it follows that $c$ is regularly recurrent. In fact, any kneading sequence that can be constructed through an infinite sequence of left proper constant length substitutions belongs to a unimodal map with regularly recurrent turning point. %Lastly, by construction, $Q(k)\in \{0,1,2\}$ for all $k\in \N$
    Lastly, as $A_1$ and $B_1$ were both chosen such that any possible concatenation prohibits large gaps between cutting times (since each begins and ends with $1$), it follows that $Q(k)\in \{0,1,2\}$ for all $k\in \N$.   
    We thus conclude that our unimodal map is longbranched and has a regularly recurrent turning point.
\end{proof}

\begin{remark}
    Note that $f|_{\omega(c)}$ is not a minimal homeomorphism in this case. If $\sigma^k(\mathcal{K}(f))$ begins with $100111$, then $k \equiv 0 \mod 6$. Since $\sigma^k(\mathcal{K}(f))$ begins with a 1 for all $k\equiv 5 \mod 6$, the only preimage of $\mathcal{K}(f)$ in the shift space generated by $\mathcal{K}(f)$ is $1\mathcal{K}(f)$. By Theorem~\ref{Theorem:Alvin_homeo}, $f|_{\omega(c)}$ is not a minimal homeomorphism.
\end{remark}

Observe that Example~\ref{Example: reg rec and longbranched} has a linearly recurrent kneading sequence, and thus by Proposition~\ref{prop: linearly recurrent not slowly recurrent} it does not have a slowly recurrent kneading sequence. The question remains whether this map satisfies the CE-condition; this is Question~\ref{Question: regularly recurrent, not slowly recurrent but CE}.

\begin{comment}
We thus end this section with the following related question.

\begin{question}
    Is it possible for a map $f$ with regularly recurrent turning point $c$ to be such that $\mathcal{K}(f)$ is not slowly recurrent, but $f$ still satisfies the CE-condition?
\end{question}
\end{comment}

% \Lori{I am uncertain if this comment is relevant anymore unless we want to contrast with Prop 55.}
% Because $c$ is regularly recurrent, $c$ is also uniformly recurrent. Note that $|A_n|$ is a co-cutting time for all $n\in \N$ and $\RR(|A_n|) = |A_n|+|A_{n-1}|+\cdots + |A_1|+3$. Additionally, $|A_n|-2$ is also always a co-cutting time; thus, $\widetilde{Q}\not\to\infty$, but $\widetilde{Q}$ is unbounded.

% \Lori{We should add a comment about the fact that the example in this section has linearly recurrent kneading sequence and thus is not slowly recurrent and then ask the question about whether it is CE. We could modify the question to ask if every map $f$ with regularly recurrent turning point $c$ such that $\mathcal{K}(f)$ is not slowly recurrent does not have the CE-condition. }
%
%
%
%
%
%
%
%
%

%
%
%
%
%
%
%
%
%
%
%
%
%
\section{Existence of a regularly recurrent unimodal map with minimal homeomorphism}\label{sec: regularly recurrent minimal homeomorphism not odometer}

It is known that unimodal maps $f$ with embedded odometers are such that $c$ is regularly recurrent and $f|_{\omega(c)}$ is a minimal homeomorphism. We demonstrate that there exist unimodal maps with regularly recurrent turning points such that $f|_{\omega(c)}$ is a minimal homeomorphism that is not conjugate to an odometer.

\begin{example}\label{Example: reg rec homeo}
    Let $A_1 = 10001$, $B_1 = 10101$, and $C_1 = 10100$. For all $n\geq 1$, define $A_{n+1} = A_nA_nB_nB_n$, $B_{n+1} = A_nB_nC_nB_n$, and $C_{n+1} = A_nB_nC_nC_n$. Let $A = \lim_{n\to \infty} A_n$. 
\end{example}

We will demonstrate that $A$ is the kneading sequence for a unimodal map $f$ such that $c$ is regularly recurrent and $f|_{\omega(c)}$ is a homeomorphism; further, $f|_{\omega(c)}$ is not conjugate to an odometer.

\begin{lemma}
$A_k$ has even parity for all $k\in \N$; $B_k$ has even parity for all even $k$ and odd parity for all odd $k$; $C_k$ has even parity for all odd $k$ and odd parity for all even $k$.
\end{lemma}
\begin{proof}
Clearly $A_1$ and $C_1$ have even parity while $B_1$ has odd parity. Let $\#(A_k)$ denote the number of ones in $A_k$, $\#(B_k)$ the number of ones in $B_k$, and $\#(C_k)$ the number of ones in $C_k$. Then $\#(A_k) = \#(A_{k-1})+\#(A_{k-1}) + \#(B_{k-1})+\#(B_{k-1}) = 2(\#(A_{k-1}) + \#(B_{k-1}))$, which is even for all $k\geq 2$. We have that $\#(B_k) = \#(A_{k-1}) + 2\#(B_{k-1}) + \#(C_{k-1})$, and thus $B_k$ has same parity as $C_{k-1}$. Similarly, $\#(C_k) = \#(A_{k-1}) + \#(B_{k-1}) + 2\#(C_{k-1})$; thus $C_k$ has the same parity as $B_{k-1}$. We thus see that $B_2$ has even parity, and $C_2$ has odd parity. Inductively, $B_{2k}$ and $C_{2k+1}$ have even parity for all $k\in \N$ while $B_{2k+1}$ and $C_{2k}$ have odd parity for all $k\in \N$.
\end{proof}

\begin{remark}\label{rem: A_k B_k}
    We make the following observations about the words $A_k$, $B_k$, and $C_k$ relative to the parity lexicographical ordering. These properties are all straightforward to verify.
    \begin{enumerate}
        \item $A_1 \succ B_1 \succ C_1$
        \item $A_1A_1 \succ A_1B_1$
        \item $A_2 \succ C_2 \succ B_2$
        \item $A_kA_k \succ A_kB_k$ for all $k\in \N$
        \item $A_{2k+1} \succ B_{2k+1}\succ C_{2k+1}$ for all $k\in \N$
        \item $A_{2k} \succ C_{2k} \succ B_{2k}$ for all $k\in \N$
        \item $A_11 \succ \sigma^k(A_11)$ for all $1\leq k \leq 4$
        \item $A_11 \succ \sigma^k(B_11)$ for all $1\leq k \leq 4$
        \item $A_11 \succ \sigma^k(C_11)$ for all $1\leq k \leq 4$
    \end{enumerate}
\end{remark}

\begin{lemma}
For every $n\in\mathbb N$, $|A_n|=|B_n|=|C_n|=5\cdot 4^{n-1}$.
\end{lemma}
\begin{proof}
This is clear for $n=1$. Each word at level $n+1$ is obtained by concatenating
four level-$n$ words, so the lengths satisfy
$|A_{n+1}|=|B_{n+1}|=|C_{n+1}|=4|A_n|$.
\end{proof}

\begin{lemma}\label{lem: A_n1}
    For each $n\in \N$, we have that $A_n1 \succ \sigma^k(A_n1)$, $A_n1\succ \sigma^k(B_n1)$, and $A_n1\succ \sigma^k(C_n1)$ for all $1\leq k \leq 5\cdot 4^{n-1}-1$.
\end{lemma}
\begin{proof}
    By properties 7-9 of Remark~\ref{rem: A_k B_k}, we know that all three claims hold when $n=1$. Hence, suppose that we have all three claims hold for some fixed $n\geq 1$. Then $A_{n+1}1 = A_nA_nB_nB_n1$, $B_{n+1}1 = A_nB_nC_nB_n1$, and $C_{n+1}1 = A_nB_nC_nC_n1$. Additionally, by the inductive claim, we only need to check that $A_{n+1}1 \succ \sigma^{t}(A_{n+1}1)$, $A_{n+1}1\succ \sigma^t(B_{n+1}1)$ and $A_{n+1}1\succ \sigma^t(C_{n+1}1)$ for $t \in  \{5\cdot 4^{n-1}, 10\cdot 4^{n-1}, 15\cdot 4^{n-1}\}$. Observe that $\sigma^{5\cdot4^{n-1}}(A_{n+1}1) = A_nB_nB_n1$, $\sigma^{10\cdot 4^{n-1}}(A_{n+1}1) = B_nB_n1$, and $\sigma^{15\cdot 4^{n-1}}(A_{n+1}1) = B_n1$. By properties 4-6 of Remark~\ref{rem: A_k B_k}, we conclude that $A_{n+1}1\succ \sigma^k(A_{n+1}1)$ for all $1\leq k\leq 5\cdot 4^n-1$. Similar logic shows that $A_{n+1}1\succ \sigma^k(B_{n+1}1)$ and $A_{n+1}1\succ \sigma^k(C_{n+1}1)$ for all $1\leq k\leq 5\cdot 4^n-1$. By induction, we conclude that all three claims hold for every $n\in \N$.
\end{proof}

\begin{proposition}
    The sequence $A = \lim_{n\to \infty}A_n$ as defined in Example~\ref{Example: reg rec homeo} is shift-maximal.
\end{proposition}
\begin{proof}
    % This follows immediately from the fact that for each $n\in \N$, $A_n1 \succ \sigma^k(A_n1)$ for all $1\leq k \leq 5\cdot 4^{n-1}-1$. \Jernej{Do we need some more explanation here? Maybe the following:}
    Let $m\geq 1$ be arbitrary. Choose $n$ sufficiently large so that $m<|A_n|=5\cdot 4^{n-1}$. Since $A_n$ is an initial block of $A$, the first $|A_n|+1-m$ symbols of $A$ and $\sigma^m(A)$ coincide with the comparison between $A_n1$ and $\sigma^m(A_n1)$. By Lemma~\ref{lem: A_n1}, $A_n1 \succ \sigma^m(A_n1)$, hence $A\succ \sigma^m(A)$. Since $m$ was arbitrary, $A$ is shift-maximal.
\end{proof}

By construction, $A$ is a Toeplitz sequence and therefore uniformly recurrent.

\begin{theorem}\label{thm: regularly recurrent minimal not odometer}
    The unimodal map $f$ with kneading sequence $A$ as defined in Example~\ref{Example: reg rec homeo} has regularly recurrent turning point $c$ and $f|_{\omega(c)}$ is a minimal homeomorphism and the map admits no embedded odometer.
\end{theorem}
\begin{proof}
    Let $A = a_1a_2a_3\cdots$ be the kneading sequence of our map. By \cite[Corollary 4.6]{Alvin_uniformly_recurrent}, we conclude that $c$ is regularly recurrent.  We note that $A$ is an aperiodic Toeplitz kneading sequence whose essential periods are
$p_i=5\cdot 4^{i-1}$ for $i\in\mathbb N$. Observe that $A_1 = 10001$, $B_1 = 10101$, and $C_1 = 10100$ imply that $a_{1+5n} = 1$, $a_{2+5n} = 0$, and $a_{4+5n} = 0$ for all $n\in \N$; that is, the only holes that appear in the $5$-skeleton of $A$ (i.e., the positions that are not periodic with period $5$) occur in positions congruent to 3 or 0 modulo 5.
    Because $A_2 = A_1A_1B_1B_1$, $B_2 = A_1B_1C_1B_1$, and $C_2 = A_1B_1C_1C_1$, we have that the only holes that appear in the $20$-skeleton of $A$ (i.e., the positions that are not periodic with period $20$) occur in the positions congruent to 8, 15, or 0 modulo 20. Additionally, the only holes that appear in the $80$-skeleton of $A$ occur in the positions congruent to 28, 35, 60, and 0 modulo 80. In particular, observe that for every $p_i = 5\cdot 4^{i-1}$, there will always be a hole in the $p_i$-skeleton of $A$ at the position congruent to $3+\sum_{k=1}^{i-1} p_k$ modulo $p_i$. 
    Thus, for every essential period $p_i$, the $p_i$-skeleton contains a hole in the congruence class $3+\sum_{k=1}^{i-1}p_k \pmod{p_i}$. Therefore the hypotheses of \cite[Theorem 4.2]{Alvin_ftcp} are not satisfied, and the unimodal map does not have an embedded odometer.

    Consider the substitution $\theta:\{A_1,B_1,C_1\}\to\{A_1,B_1,C_1\}^\ast$ given by $\theta(A_1) = A_1A_1B_1B_1$, $\theta(B_1) = A_1B_1C_1B_1$, and $\theta(C_1) = A_1B_1C_1C_1$. Observe that, from the allowed predecessor words determined by $\theta$, the
only point in the associated one-sided shift space $X_\theta$ that has two
preimages is the fixed point $\lim_{n\to\infty}\theta^n(A_1)$.
Further note that one preimage begins with $B_1$ and the other with $C_1$. In Example~\ref{Example: reg rec homeo}, for all $n\geq 2$, $A_n = \theta^{n-1}(A_1)$, $B_n = \theta^{n-1}(B_1)$, and $C_n = \theta^{n-1}(C_1)$.
    
    In the construction of Example~\ref{Example: reg rec homeo}, let $X_f$ denote the closure of the forward shift of the kneading sequence $\mathcal{K}(f) = A$ and recall Theorem~\ref{Theorem:Alvin_homeo}.

    Note that $B_1 = 10101$ and $C_1=10100$ agree in all but the last position. Also note that $B_{n+1} = A_nB_nC_nB_n = \theta^n(B_1)$ and $C_{n+1} = A_nB_nC_nC_n = \theta^n(C_1)$, and thus as sequences of 0s and 1s, $B_{n+1}$ and $C_{n+1}$ agree in all but the last position for each $n\in \N$. Observe that any arbitrarily long initial block of $\mathcal{K}(f)$ begins with $\theta^n(A_1) = A_{n+1}$. Because the words $A_1=10001$, $B_1=10101$, and $C_1=10100$ have the same length, share the common prefix $10$, and differ only in their final two symbols, every occurrence of one of these words in a point of $X_f$ is uniquely determined by its initial position modulo $5$. Hence, every point of $X_f$ admits a unique decomposition into level-$1$ words, and inductively into $\theta^n$-words. Thus, each occurrence of $A_{n+1}$ is preceded by either $A_{n+1}$, $B_{n+1}$, or $C_{n+1}$.
    We emphasize that $A_{n+1} = A_nA_nB_nB_n$, $B_{n+1} = A_nB_nC_nB_n$, and $C_{n+1} =A_nB_nC_nC_n$. Hence, every occurrence of $A_{n+1}$ is preceded by either $B_n$ or $C_n$, which as sequences of 0s and 1s agree in every position but the last. We conclude that $\mathcal{K}(f)$ has two preimages in $X_f$ and whenever $a0\mathcal{K}(f),a'1\mathcal{K}(f) \in X_f$ with $|a| = |a'|< \infty$, it follows that $a = a'$.

    Now suppose that there is another point $x\in X_f$ with two preimages under the shift. Because the decomposition of $x$ into its $\theta$-words is unique, we have that $x = \sigma^k(d)w$ where $d\in \{10001, 10101, 10100\}$, $0\leq k < 5$, and $w\in X_f$ is such that the recoding into $\{A_1,B_1,C_1\}^\N$ is exactly a point of $X_\theta$. We let $\hat{w}$ denote the recoding of $w$ into $\{A_1,B_1,C_1\}^\N$ such that $\hat{w} \in X_\theta$ and let $\hat{d}$ represent the appropriate word $A_1,B_1,C_1$ for $d$. If $k = 0$, then $\hat{d}\hat{w}\in X_\theta$ has two preimages, and hence $\hat{d}\hat{w}=\lim_{n\to \infty}\theta^n(A_1)$; therefore $x = \mathcal{K}(f)$. Hence suppose that $1\leq k < 5$. This implies that $\hat{w}$ has two preimages in $X_\theta$, but then $\hat{w} = \lim_{n\to\infty}\theta^n(A_1)$. Hence $\hat{d}\in \{B_1,C_1\}$, and since $B_1 = 10101$  and $C_1 = 10100$, it is not possible for such a $1\leq k < 4$ to exist.  We thus conclude that there is no other point $x\in X_f$ with two preimages under the shift. Hence, the shift on $X_f$ is one-to-one. Since $A$ is uniformly recurrent, $X_f$
is minimal, and therefore by Theorem~\ref{Theorem:Alvin_homeo},
$f|_{\omega(c)}$ is a minimal homeomorphism.
\end{proof}

In investigating the kneading and co-kneading maps for Example~\ref{Example: reg rec homeo}, we note that if we decompose the kneading sequence $A$ into a sequence of the form $\{A_n,B_n,C_n\}^\N$ (for any $n\in \N$), the following hold.
\begin{itemize}
    \item If a cutting time appears in the second position of a word $A_n$, then there will be a cutting time in the second position of the next word.
    \item If a cutting time appears in the second position of a word $B_{2k+1}$ (or $C_{2k}$), then there will be a cutting time at the end of that word.
    \item If a cutting time appears at the end of the word immediately before $B_{2k+1}$ (or $C_{2k}$), then there will be a cutting time at the second position of the next word.
    \item If a cutting time appears in the second position of a word $B_{2k}$ (or $C_{2k+1}$), then there will be a cutting time in the second position of the next word.
    \item If a cutting time appears at the end of the word immediately before $B_{2k}$ (or $C_{2k+1}$), then there will be a cutting time at the end of the $B_{2k}$ (or $C_{2k+1}$).
\end{itemize}

By using these facts, we observe that for all $n\geq 2$, the following hold.
\begin{itemize}
    \item $|A_n|$ is a co-cutting time and $\mathcal{R}(|A_n|) = |A_n| + |A_{n-1}|+\cdots + |A_1| + 3$.
    \item If $n$ is odd, then $|A_nA_nB_n|=3|A_n|$ is a cutting time and $\mathcal{R}(3|A_n|)=|A_{n-1}|+\cdots + |A_1| + 3$.
    \item If $n$ is even, then $|A_nA_nB_n|=3|A_n|$ is a co-cutting time and $\mathcal{R}(3|A_n|)=|A_{n-1}|+\cdots + |A_1| + 3$.
\end{itemize}

We thus conclude that both the kneading map $Q(k)$ and the co-kneading map $\widetilde{Q}(k)$ are unbounded and $\liminf Q(k) = \liminf \widetilde{Q}(k)=0$. %\Lori{Is it better to use the linear recurrence argument for the next statement? Does it matter?} \Jernej{I think it can stay as it is.} 
Since $\mathcal{R}(|A_n|)>|A_n|$ for all $n\in \N$, we may apply Lemma~\ref{Lemma:Long co-cuts not slowly recurrent} to obtain that this example is not slowly recurrent; we do not know whether this example satisfies the CE-condition. 
Lastly, the unimodal map does not have a persistently recurrent turning point: for every $\epsilon > 0$, there exists an $n$ large enough such that $c_{|A_n|}\in B(c,\epsilon)$ and $\widetilde{D}_{|A_n|} = [c_3,c_1]$.

\section{Diagram of inclusions and open questions}

Since several different notions of recurrence and combinatorial behavior appear throughout the paper, we include the schematic diagram in Figure~\ref{fig: the diagram} illustrating the currently known relations among the nine classes that are most relevant for the present work. We deliberately restrict attention only to these nine notions, since including further conditions would make the picture significantly more complicated and substantially harder to interpret. The diagram should be understood only on the level of set-theoretic inclusions and possible intersections; it is not intended to reflect topological, measure theoretic or density properties.

The black dots in the diagram indicate regions where no explicit examples are currently known to the authors, but where we strongly suspect that examples should exist. Likewise, question marks indicate relationships whose precise status remains unclear and we state them as open questions later in this section.

The largest class under consideration is the class of infinitely recurrent unimodal maps. Inside this class lie uniformly recurrent maps (condition (C9)) and, more restrictively, persistently recurrent maps ((C8)$\iff$(C7)). Persistent recurrence imposes strong combinatorial restrictions on the critical orbit and is closely related to the structure of return times and the kneading map. Another important subclass consists of maps for which the turning point is not critically monotonic (condition (C6)); this condition is closely connected with the existence of wild attractors.

On the more topological side, we consider regularly recurrent systems, minimal homeomorphisms, and strange adding machines. These classes are closely related to odometer-type dynamics, although the precise extent of this relationship in the unimodal setting is not yet fully understood. Finally, we also include longbranched maps, i.e.\ maps with bounded kneading map $Q$, whose combinatorics are highly constrained but still appear compatible with a surprisingly broad range of dynamical behaviors. The unresolved regions in Figure~\ref{fig: the diagram} naturally lead to several open problems. 

\begin{comment}
\begin{question}\label{Question: wild attractors}
Suppose that $f$ is an infinitely recurrent unimodal map such that the turning point is not critically monotonic (i.e.\ condition (C6) holds). Must $f$ admit a wild attractor?
\end{question}

\begin{question}\label{Question: longbranched}
Let $f$ be a longbranched infinitely recurrent unimodal map. Is it possible that $f|_{\omega(c)}$ is a minimal homeomorphism or even a strange adding machine?
\end{question}

\begin{question}\label{Question: persistently recurrent minimal}
Is there an example of a persistently recurrent unimodal map such that $f|_{\omega(c)}$ is a minimal homeomorphism which is not conjugate to an odometer?
\end{question}

\begin{question}\label{Question: regularly recurrent persistently recurrent}
Is there an example of a persistently recurrent unimodal map with regularly recurrent turning point that is not conjugate to a strange adding machine?
\end{question}
\end{comment}

\begin{figure}[ht]
\centering

\begin{tikzpicture}[x=1.6cm,y=1.5cm]

% colors
\definecolor{cInf}{RGB}{226,102,102}
\definecolor{cPer}{RGB}{ 10,110,255}
\definecolor{cUni}{RGB}{164,  0,220}
\definecolor{cReg}{RGB}{245, 105, 30}
\definecolor{cQ}{RGB}{235,150,  0}
\definecolor{cRen}{RGB}{150,100, 45}
\definecolor{cFib}{RGB}{150,150,150}
\definecolor{cLong}{RGB}{ 95,185, 65}
\definecolor{c5}{RGB}{ 20, 20,140}
\definecolor{cMono}{RGB}{  0,125, 70}
\definecolor{cSAM}{RGB}{ 90,205,195}
\definecolor{cMin}{RGB}{255,  0,0}
\definecolor{cSlow}{RGB}{  0,190,235}
\definecolor{cCE}{RGB}{100,100,100}

\tikzset{maincurve/.style={line width=1.0pt}}

% outer family
\draw[cInf,maincurve] (-0.1,-0.3) ellipse [x radius=5.00, y radius=3];
\draw[cUni,maincurve] (-0.08,-0.33) ellipse [x radius=4.28, y radius=2];
\draw[cPer,maincurve] (1.2,-0.08) ellipse [x radius=1.98, y radius=1.2];

% inner cluster inside blue
\draw[cMono,maincurve] (1.50,-0.17) ellipse [x radius=1.2, y radius=1];
%\draw[c5,maincurve] (0.41,-0.15) ellipse [x radius=0.9, y radius=0.68];
\draw[black,maincurve] (1.75,-0.195) ellipse [x radius=0.62, y radius=0.9];

% left-side ellipses
\draw[cReg,maincurve] (-0.5,-0.2) ellipse [x radius=2.6, y radius=1.15];
\draw[cLong,maincurve,rotate around={-10:(-1.50,1)}]
  (-3.1,-0.1) ellipse [x radius=2.2, y radius=0.4];
\draw[cMin,maincurve,rotate around={0:(-1.75,-0.30)}]
  (-0.6,-0.2) ellipse [x radius=2.9, y radius=0.55];
\draw[cSAM,maincurve,rotate around={0.5:(-1.68,-0.3)}]
  (0.02,-0.3) ellipse [x radius=1.8, y radius=0.28];
%\draw[cSlow,maincurve,rotate around={-8:(-1.22,-1.57)}]  (-1.22,-1.7) ellipse [x radius=0.4, y radius=1.2];

%\draw[cCE,maincurve,rotate around={-8:(-1.22,-1.57)}] (-1.22,-1.7) ellipse [x radius=0.7, y radius=1.38];

%\draw[cQ,line width=3pt,rotate around={4:(-0.22,-0.31)}] (-0.15,-0.2) ellipse [x radius=0.30, y radius=0.47];

%\draw[cFib,line width=1.5pt]  (-0.21,0.05) ellipse [x radius=0.15, y radius=0.15];

%\draw[cRen,line width=1.5pt] (-0.20,-0.4) ellipse [x radius=0.095, y radius=0.095];

\node at (-1.35,-0.215) {\small ?};
\node at (-2.1,0) {\small ?};
\node at (-3.2,-.1) {\small ?};
%\node at (-1.15,-0.7) {\tiny ?};
%\node at (-0.7,-1.1) {\small ?};
\node at (-0.15,0.18) {\small ?};
%\node at (-0.1,-0.3) {\small ?};
%\node at (0.7,-0.25) {\small ?};
%\node at (0.1,0.4) {\tiny ?};
%\node at (1.45,-0.3) {\small ?};
\node at (0.7,-0.3) {\small ?};
\node at (0.7,0.11) {\small ?};
%\node at (1.73,-0.3) {\tiny ?};
\node at (1.43,0.05) {\small ?};
%\node at (1.45,0.02) {\tiny ?};
\fill (2.6,0.5) circle (3pt);
\fill (2.4,0.3) circle (3pt);
%\fill (1.8,-0.9) circle (3pt);
\fill (-1,-1) circle (3pt);
\fill (1,0.5) circle (3pt);
\fill (0,0.6) circle (3pt);
\fill (1.5,0.3) circle (3pt);

\node at (-3.3,-2.15) {\color{cInf} \small 1};
\node at (-2.5,-1.5) {\color{cUni} \small 2};
\node at (2.9,-0.1) {\color{cPer} \small 3};
\node at (-2.2,-0.9) {\color{cReg} \small 6};
\node at (2.5,-0.1) {\color{cMono} \small 4};
\node at (1.8,0.5) {\color{black} \small 5};
\node at (-2,-0.5) {\color{cMin} \small 7};
\node at (-0.1,-0.3) {\color{cSAM} \small 8};
\node at (-5.1,0.6) {\color{cLong} \small 9};
%\node at (-1.3,-1.6) {\color{cSlow} \small 13};
%\node at (-1.77,-1.6) {\color{cCE} \small 14};

\end{tikzpicture}

\caption{\small
\textcolor{cInf}{1 - Infinitely recurrent}; \
\textcolor{cUni}{2 - uniformly recurrent (C9)}; \
\textcolor{cPer}{3 - persistently recurrent ((C8)$\iff$ (C7))}; \
\textcolor{cMono}{4 - not critically monotonic (C6)}; \
%\textcolor{c5}{5 - Condition (C5)}; \
\textcolor{black}{5 - wild attractor}; \
%\textcolor{cQ}{7 - $Q(k)\to\infty$ (C4)}; \
%\textcolor{cFib}{Fibonacci like (C1)}; \
%\textcolor{cRen}{infinitely renormalizable}; \
\textcolor{cReg}{6 - regularly recurrent}; \
\textcolor{cMin}{7 - minimal homeomorphism}; \
\textcolor{cSAM}{8 - strange adding machine}; \
\textcolor{cLong}{9 - longbranched $(Q(k)$ bounded)}.\\ We strongly believe that in the regions with $\bullet$, modifications of known examples give required examples, however are not, to our knowledge, provided in the literature yet.}
%\textcolor{cSlow}{slowly recurrent}; \
%\textcolor{cCE}{Collet-Eckmann condition}.} 
\label{fig: the diagram}
\end{figure}

The first question addresses the extent to which combinatorial constraints on the kneading map determine the existence of wild attractors (absorbing Cantor sets). It is known that (C6) is necessary, and that stronger conditions such as (C5) or (C4) impose additional control on return times. However, it remains open whether these conditions are sufficient in the presence of embedded odometer dynamics, that is, when $f|_{\omega(c)}$ is a minimal surjection that is not a homeomorphism.

\begin{question}\label{Question: wild attractors}
    Does every kneading sequence belonging to a unimodal map with an embedded strange odometer for which the turning point is not critically monotonic (i.e., (C6) holds) have combinatorics that allows for the existence of an absorbing Cantor set as long as the associated map has high enough order? If not, what if we assume either (C5) or (C4) holds?
\end{question}

The second question concerns the compatibility of longbranched combinatorics with minimal dynamics. Boundedness of the kneading map severely restricts the combinatorics of cutting times, yet it is not known whether such constraints exclude the possibility of minimal homeomorphisms or adding machine dynamics on the omega-limit set.

\begin{question}\label{Question: longbranched}
Let $f$ be a longbranched (infinitely) recurrent unimodal map. Is it possible that $f:\omega_{f}(c) \to \omega_f(c)$ is a minimal homeomorphism or even a strange adding machine?
\end{question}

The third question concerns the structure of minimal dynamics under strong recurrence assumptions. While regularly recurrent maps are typically associated with odometer-type behavior, it remains open whether more general minimal homeomorphisms can arise in this setting when persistent recurrence is imposed.

\begin{question}\label{Question: regularly recurrent, minimal, not odometer and persistently recurrent}
    Is there an example of a unimodal map $f$ that has regularly recurrent turning point, $f|_{\omega(c)}$ is a minimal homeomorphism that is not conjugate to an odometer, and $c$ is persistently recurrent?
\end{question}

The final question explores the interaction between regular recurrence and non-uniform hyperbolicity. Regular recurrence enforces strong recurrence properties of the critical orbit, while the Collet--Eckmann condition requires exponential growth along derivatives. It is unclear whether these two types of behavior can coexist without additional slow recurrence assumptions.

\begin{question}\label{Question: regularly recurrent, not slowly recurrent but CE}
    Is it possible for a map $f$ with regularly recurrent turning point $c$ to be such that $\mathcal{K}(f)$ is not slowly recurrent, but $f$ still satisfies the CE-condition?
\end{question}

\section{Acknowledgments}
L.~Alvin was partially supported by the Fulbright U.S. Scholar Program and the Henry Keith and Ellen Hard Townes endowed professorship at Furman University.
 J.~\v Cin\v c was partially supported by Slovenian research agency ARIS grant J1-4632 and ARIS project under Contract No. SN-ZRD/22-27/0552. 

\begin{table}[ht]
\begin{tabular}[t]{p{2.5cm}  p{11cm} }
\includegraphics [width=2.1cm]{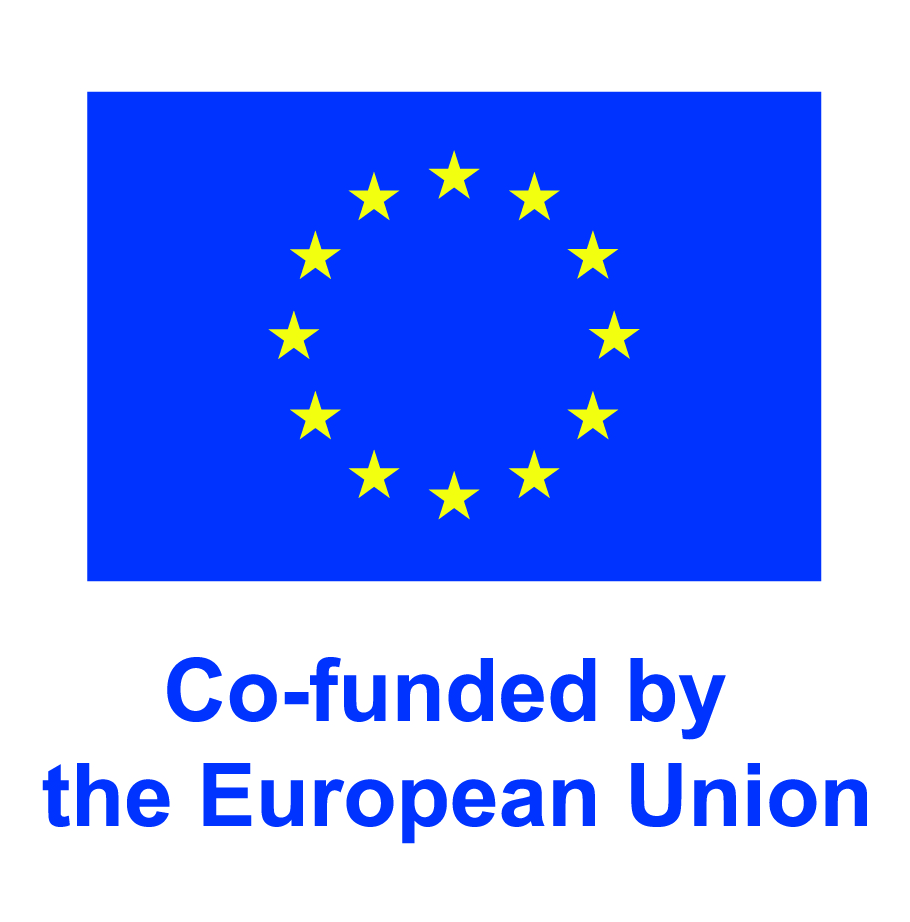} &
\vspace{-2cm}
This research is part of J.\ \v Cin\v c's  project that has received funding from the European Union's Horizon Europe research and innovation programme under the Marie Sk\l odowska-Curie grant agreement No.\ HE-MSCA-PF-PFSAIL-101063512.\\
\end{tabular}
\end{table}

\newpage

\bibliographystyle{plain}
\bibliography{BIBpersist} 
\end{document}